\documentclass[10pt,leqno]{amsart}
\usepackage{graphicx}
\usepackage{indentfirst,csquotes}

\usepackage{amssymb,amsthm,amsmath,amsfonts,mathrsfs}
\usepackage{xcolor,paralist,hyperref,fancyhdr,etoolbox}
\usepackage{amscd}
\usepackage{tikz-cd}
\newtheorem{theorem}{Theorem}[section]
\newtheorem{definition}[theorem]{Definition}

\newtheorem{lemma}[theorem]{Lemma}
\newtheorem{proposition}[theorem]{Proposition}
\newtheorem{corollary}[theorem]{Corollary}

\hypersetup{ colorlinks=true, linkcolor=black, filecolor=black, urlcolor=black }

\usepackage{lipsum}

\begin{document}
\title[Structure of derived Hecke algebras with coefficients in profinite rings]{Structure of derived Hecke algebras with coefficients in torsion-free profinite rings}
\author[Shuta Kataoka]{Shuta Kataoka}
\date{}
\address{Department of Mathematics,
	Graduate School of Science,
	The University of Osaka}
\address{National Institute of Technology, Oshima College}
\email{math.shutakataoka@gmail.com}

% !TeX root = main.tex
\begin{abstract}
	In this article,
	we study derived Hecke algebras with coefficients in torsion-free profinite rings regarded as discrete rings.
	We compare each such algebra with the degreewise inverse limit of the derived Hecke algebras with coefficients in finite rings 
	and with its subalgebra consisting of elements with uniformly finite double-coset support.
	Under cohomological comparison hypotheses,
	the natural map to this subalgebra has image given by its degree-zero part
	and its positive-degree torsion.
	We describe the kernel in terms of continuous cohomology and prove that it is a divisible square-zero ideal annihilated by every positive-degree element.
	These comparisons are compatible with Yoneda products,
	and the action on torsion Ext classes factors through the image.
	We compute the resulting algebras for the additive group $\mathbb{Q}_p$ and determine the underlying graded abelian groups for $(\operatorname{GL}_2(\mathbb{Q}_p),\operatorname{GL}_2(\mathbb{Z}_p))$ with coefficients in $\mathbb{Z}_p$,
	for $p>3$.
	The examples exhibit distinct effects of the coefficient topology and the support condition.
\end{abstract}
	
\maketitle
\tableofcontents
\let\thefootnote\relax

\bigskip

$\,$

% !TeX root = main.tex
\section{Introduction}

Let $R$ be a commutative ring,
let $G$ be a locally profinite group,
and let $K$ be an open compact subgroup of $G$.
We give $R$ the trivial $G$-action.
The category $\mathrm{SM}_R(G)$ of smooth $R[G]$-modules is a Grothendieck category and hence has enough injectives.
Using Ext in this category,
we define
\[
	\mathscr H(G,K)_R
		:=\bigoplus_{i\geqq 0}\mathscr H^i(G,K)_R
		:=\bigoplus_{i\geqq 0}
		\operatorname{Ext}^i_{\mathrm{SM}_R(G)}(R[G/K],R[G/K])
\]
and call this the derived Hecke algebra for $(G,K)$ with coefficients in $R$
\cite[Definition 2.2]{Venkatesh}.
Its multiplication is given by the Yoneda product,
which we denote by $*$.

In the local setting of \cite[Sections 2.1--2.3]{Venkatesh},
Venkatesh describes this product using functions with values in group cohomology,
for finite coefficient rings in which the residue characteristic of the local field is invertible.
For arbitrary locally profinite groups and commutative coefficient rings,
Kozio{\l},
Ollivier,
and Stockton give a formula for the opposite of the Yoneda product in terms of group cohomology \cite[Proposition 2.10]{koziol2024parahoricheckeextalgebrascharacteristic}.
Venkatesh also defines a global derived Hecke algebra with coefficients in $\mathbb Z_\ell$ using compatible systems of operators on arithmetic cohomology with finite coefficients \cite[Section 2.13]{Venkatesh}.

In this paper,
we study local derived Hecke algebras before passing to the image of an action on arithmetic cohomology.
More precisely,
we study derived Hecke algebras with coefficients in torsion-free profinite rings equipped with the discrete topology,
and compare them with inverse limits of derived Hecke algebras with finite coefficients.
Under the cohomological hypotheses stated below,
we describe the image and kernel of the natural comparison map and their relation to the Yoneda product.

Write $R^d$ for $R$ with the discrete topology.
The Ext description gives an isomorphism of graded $R$-modules
\[
	\mathscr H(G,K)_R
		\cong \bigoplus_{x\in K\backslash G/K}H^*(K_x,R^d),
\]
where $g_x\in G$ is a representative of $x\in K\backslash G/K$ and
\[
	K_x:=K\cap g_xKg_x^{-1}.
\]
Thus the study of the underlying graded module reduces to continuous cohomology of profinite groups with discrete coefficients.

The foundation of our results is a comparison of continuous cohomology with coefficients in a torsion-free profinite abelian group $A$.
Fix a presentation $A\cong \varprojlim_{i\in I}A_i$ as a topological group,
where $(A_i)_{i\in I}$ is an inverse system of finite abelian groups.
Write $A^d$ and $A^p$ for $A$ with the discrete and profinite topologies,
respectively.
All group cohomology considered below is continuous,
and all coefficient actions are trivial.
We compare the map 
\[
	H^*(K,A^d)\rightarrow H^*(K,A^p)
\]
induced by the identity $A^d\rightarrow A^p$ with the map $H^*(K,A^d)\rightarrow \varprojlim_{i\in I} H^*(K,A_i)$ induced by the coefficient maps $A\rightarrow A_i$.
Under suitable hypotheses on the comparison with rationalization,
we identify the image of the natural map $H^*(K,A^d)\rightarrow H^*(K,A^p)$ in positive degree with the torsion subgroup of group cohomology with profinite coefficients (Proposition \ref{proposition1-1}).

When $A$ is a commutative profinite ring,
we also describe the kernel and the cup product under the hypotheses of Corollary \ref{theorem2}.
These general comparison results do not require the comparison map from cohomology with profinite coefficients to the inverse limit of cohomology with finite coefficients to be an isomorphism.

Applying these results to each $K_x$,
and using their compatibility with the restriction,
corestriction,
and conjugation maps appearing in the Hecke product,
we obtain the main theorem on derived Hecke algebras stated below.

Suppose from now on that $R$ is a torsion-free profinite ring,
and write $R^p$ when we equip it with its profinite topology.
Fix a presentation
\[
	R^p\cong\varprojlim_{i\in I}R_i
\]
as an inverse limit of finite commutative rings.
For every $n\geqq 0$,
define
\begin{align*}
	\overline{\mathscr{H}}^{\,n}(G,K)_R
		&:=\bigoplus_{x\in K\backslash G/K}
		\varprojlim_{i\in I}H^n(K_x,R_i),\\
	\widehat{\mathscr{H}}^{\,n}(G,K)_R
		&:=\varprojlim_{i\in I}\mathscr{H}^n(G,K)_{R_i}.
\end{align*}

Let $\overline{\mathscr{H}}(G,K)_R$ and $\widehat{\mathscr{H}}(G,K)_R$ be the direct sums of their respective homogeneous components.
If $(h_i)_{i\in I}$ is an element of $\widehat{\mathscr{H}}(G,K)_R$,
then the support of each $h_i$ is a finite subset of $K\backslash G/K$.
The union of these supports over all $i$,
however,
need not be finite.
The elements for which this union is finite are precisely those arising from $\overline{\mathscr{H}}(G,K)_R$.

The Yoneda products with finite coefficients induce a product on $\widehat{\mathscr{H}}(G,K)_R$,
and $\overline{\mathscr{H}}(G,K)_R$ is a subalgebra of $\widehat{\mathscr{H}}(G,K)_R$.
We denote both products by $*_p$.
The natural map $\mathscr{H}(G,K)_R\rightarrow\widehat{\mathscr{H}}(G,K)_R$ is an algebra homomorphism that factors through $\overline{\mathscr{H}}(G,K)_R$ (Propositions 5.1 and 5.6).
We denote the resulting map $\mathscr{H}(G,K)_R\rightarrow\overline{\mathscr{H}}(G,K)_R$
by $\rho$.

For the main theorem,
we impose the following topological and cohomological hypotheses.
Give $\mathbb{Q}$ the discrete topology and $R^p\otimes_{\mathbb{Z}}\mathbb{Q}$ the direct limit topology.
Denote the resulting space by $(R\otimes_{\mathbb{Z}}\mathbb{Q})^p$,
and assume that it is a topological $\mathbb{Q}$-vector space.
We further assume that the natural inclusion $R^p\hookrightarrow(R\otimes_{\mathbb{Z}}\mathbb{Q})^p$ is an open embedding.

For every $n\geqq 1$ and $x\in K\backslash G/K$,
we also assume that the natural maps
\begin{align*}
	H^n(K_x,R^p)\otimes_{\mathbb{Z}}\mathbb{Q}
		&\longrightarrow H^n(K_x,(R\otimes_{\mathbb{Z}}\mathbb{Q})^p),\\
	H^n(K_x,R^p)
		&\longrightarrow \varprojlim_{i\in I}H^n(K_x,R_i)
\end{align*}
are isomorphisms.
For an abelian group $A$,
write $A_{\mathrm{tor}}$ for its torsion subgroup.

\begin{theorem}[Proposition \ref{proposition1-1} and Theorem \ref{Theorem4}]
	Under the hypotheses above,
	set $X:=\mathrm{Im}\rho$ and $Y:=\ker\rho$,
	and denote their homogeneous components by $X^n$ and $Y^n$.
	Then
	\[
		X^0=\overline{\mathscr{H}}^{\,0}(G,K)_R,
			\qquad
			X^n=
			\bigl(
				\overline{\mathscr{H}}^{\,n}(G,K)_R
			\bigr)_{\mathrm{tor}}
			\quad(n\geqq 1).
	\]
	In particular,
	$X^1=\mathscr{H}^1(G,K)_R=0$.
	Moreover,
	$Y^0=Y^1=0$,
	and for every $n\geqq 2$ there is a natural isomorphism of abelian groups
	\[
		Y^n\cong
		\bigoplus_{x\in K\backslash G/K}
		\left(
			H^{n-1}(K_x,R^p)\otimes_{\mathbb{Z}}(\mathbb{Q}/\mathbb{Z})
		\right).
	\]
	Thus each $Y^n$ is divisible,
	and the short exact sequence
	\[
		0\longrightarrow Y
			\longrightarrow\mathscr{H}(G,K)_R
			\xrightarrow{\rho}X
			\longrightarrow0,
	\]
	obtained by restricting the codomain of $\rho$ to $X$,
	splits as a sequence of graded abelian groups.
	Furthermore,
	writing $\mathscr{H}^{>0}(G,K)_R:=\bigoplus_{n>0}\mathscr{H}^n(G,K)_R$,
	we have
	\[
		\mathscr{H}^{>0}(G,K)_R*Y=Y*\mathscr{H}^{>0}(G,K)_R=0.
	\]
	In particular,
	$Y$ is a square-zero ideal.
\end{theorem}

This theorem describes the subgroup $Y$ of elements that map to zero under every map $\mathscr{H}(G,K)_R\rightarrow\mathscr{H}(G,K)_{R_i}$ induced by a coefficient map $R\rightarrow R_i$.
In particular,
$Y$ is divisible as an abelian group.
The splitting is a splitting of graded abelian groups and need not respect multiplication.
Theorem \ref{Theorem4} gives a formula for the product after choosing a degree-preserving additive section.

Independently of the main theorem,
we also obtain a description of derived Hecke algebras with coefficients in torsion-free profinite rings as direct limits.
Specifically,
suppose that $(S_i)_{i\in I}$ is a filtered direct system of finite abelian groups such that
\[
	(R\otimes_{\mathbb{Z}}\mathbb{Q})/R\cong\varinjlim_{i\in I}S_i
\]
as abelian groups,
and that each $S_i$ is equipped with the structure of a commutative ring.
Then,
for every $n\geqq 2$,
there is an isomorphism of abelian groups
\[
	\mathscr{H}^{n}(G,K)_R\cong\varinjlim_{i\in I}\mathscr{H}^{n-1}(G,K)_{S_i}.
\]
Here the direct limit is taken in the category of abelian groups.

Finally,
we compute derived Hecke algebras for the additive group $G=(\mathbb{Q}_p,+)$ with coefficients in $\mathbb{Z}_p$,
$\mathbb{Z}_\ell$ for $\ell\neq p$,
and $\widehat{\mathbb{Z}}$.
For $p>3$,
we also determine the underlying graded abelian group of the derived Hecke algebra for $G=\operatorname{GL}_2(\mathbb{Q}_p)$ and $K=\operatorname{GL}_2(\mathbb{Z}_p)$ with coefficients in $\mathbb{Z}_p$.
To compare these algebras with $\overline{\mathscr{H}}(G,K)_R$ and $\widehat{\mathscr{H}}(G,K)_R$,
we also include the cohomology calculations with finite coefficients and with coefficients in $\mathbb{Q}_p/\mathbb{Z}_p$.
In particular,
we obtain the following result.

\begin{proposition}[Proposition \ref{proposition5}]
	Let $p>3$.
	Set $F_0=F_1=0$ and $F_n:=\mathbb{Z}/p^{n-1}\mathbb{Z}$ for $n\geqq 2$.
	Let
	\[
		A:=
		\left\{
			x_{b,n}=
			\begin{pmatrix}
				p^{n+b}&0\\
				0&p^b
			\end{pmatrix}
			\;\middle|\;
			b\in\mathbb{Z},\ n\geqq 0
		\right\}.
	\]
	For an abelian group $M$,
	write $M[A]:=\bigoplus_{x\in A}Mx$.
	Then there are isomorphisms of abelian groups
	\[
		\mathscr{H}^i
			\bigl(
				\operatorname{GL}_2(\mathbb{Q}_p),\operatorname{GL}_2(\mathbb{Z}_p)
			\bigr)_{\mathbb{Z}_p}\cong
			\begin{cases}
				\mathbb{Z}_p[A] & i=0,\\
				0 & i=1,\\
				\displaystyle
				\bigoplus_{\substack{b\in\mathbb{Z}\\n\geqq 0}}
				\bigl(
					\mathbb{Q}_p/\mathbb{Z}_p\oplus F_n
				\bigr)x_{b,n} & i=2,\\
				\displaystyle
				\bigoplus_{\substack{b\in\mathbb{Z}\\n\geqq 0}}F_nx_{b,n} & i=3,\\
			(\mathbb{Q}_p/\mathbb{Z}_p)[A] & i=4,5,\\
			0 & i\geqq 6.
		\end{cases}
	\]
\end{proposition}

This computation also gives an obstruction to a derived Satake isomorphism with coefficients in the discrete ring $\mathbb{Z}_p$.
Let $T$ be the group of $\mathbb{Q}_p$-rational points of the diagonal split maximal torus of $\operatorname{GL}_2$,
put $K_T:=\operatorname{GL}_2(\mathbb{Z}_p)\cap T$,
and let $W\cong S_2$ be the corresponding Weyl group.
Since $K_T$ is a $p$-torsion-free compact $p$-adic Lie group of dimension $2$,
we have $\mathrm{cd}_p(K_T)=2$.
Moreover,
$H^j(K_T,\mathbb{Q}_p^d)=0$ for every $j>0$,
so the short exact sequence of discrete modules
\[
	0\longrightarrow\mathbb{Z}_p^d
		\longrightarrow\mathbb{Q}_p^d
		\longrightarrow\mathbb{Q}_p/\mathbb{Z}_p
		\longrightarrow0
\]
gives
\[
	H^i(K_T,\mathbb{Z}_p^d)\cong H^{i-1}(K_T,\mathbb{Q}_p/\mathbb{Z}_p)=0\qquad(i>3).
\]
Since $T$ is abelian,
it follows that $\mathscr{H}(T,K_T)_{\mathbb{Z}_p}$ is concentrated in degrees at most $3$.
The preceding proposition gives nonzero components in degrees $4$ and $5$ for $\mathscr{H}(G,K)_{\mathbb{Z}_p}$.
Consequently,
there is no isomorphism of graded algebras
\[
	\mathscr{H}(G,K)_{\mathbb{Z}_p}\cong
	\bigl(
		\mathscr{H}(T,K_T)_{\mathbb{Z}_p}
	\bigr)^W.
\]
The setting considered here,
with coefficients in the discrete ring $\mathbb{Z}_p$,
lies outside the coefficient hypotheses of \cite[Theorem 3.3]{Venkatesh}.

\section*{Acknowledgements}
The author is grateful to Professor Seidai Yasuda for continued advice and helpful discussions over many years.

% !TeX root = main.tex
\section{Setting}

Let $G$ be a locally profinite group,
and let $K$ be an open compact subgroup of $G$.
Let $R$ be a commutative ring,
in particular a profinite ring,
and let $A$ be a smooth $R[G]$-representation.
We denote the continuous group cohomology of $G$ with coefficients in $A$ by $H^*(G,A)$.

We fix a complete set of representatives in $G/K$ for $K\backslash G/K$.

By a $K$-injective resolution,
we mean a resolution in the homotopical sense
(\cite{stacks-project}, Part~1, Section~13.31).

Let $A=\bigoplus_i A_i$ be a graded ring.
For $a\in A_i$ and $b\in A_j$,
we define $a*_{\mathrm{gop}}b$ to be the graded opposite product,
that is,
\[
	a*_{\mathrm{gop}}b=(-1)^{ij}ba.
\]

% !TeX root = main.tex
\section{The Relationship between Torsion-Free Discrete Coefficients and Profinite Coefficients}
Throughout this section,
let $G$ be a profinite group.
As an abelian group,
a derived Hecke algebra can be defined as a direct sum of the group cohomology groups of profinite groups.
When the coefficients of group cohomology are given by the inverse limit $\varprojlim A_i$ of an inverse system $(A_i,\phi_{ij},I)$,
the cohomology groups of the modules appearing in the inverse system also form an inverse system
\[
	(H^*(G,A_i),H^*(\phi_{ij}),I)
\]
via the natural maps induced by the maps between the coefficients.
Thus,
one may also consider the inverse limit of the cohomology groups themselves,
\[
	\varprojlim H^*(G,A_i).
\]

In this section,
we first investigate the relationship between group cohomology with discrete coefficients $R$ and $r$-adic completion with respect to the ideal $(r)$ generated by a suitable element $r\in R$,
and show that this approach is not sufficient for treating profinite coefficients.
We therefore next assume that the coefficient ring is torsion-free,
investigate the structures and properties of group cohomology when the coefficients are equipped with the discrete topology or the profinite topology,
and then study the relationship between the resulting cohomology groups.

\subsection{\texorpdfstring{$r$}{r}-adic Completion of Group Cohomology with Discrete Coefficients}

Let $R$ be a commutative ring,
and let $A$ be a smooth $R[G]$-module.
Let $r\in R$,
and assume that multiplication by $r$ on $A$ is injective,
namely,
the following sequence is exact:
\[
	0\rightarrow A \xrightarrow{r} A \xrightarrow{\mod r} A/rA \rightarrow 0.
\]

Then this short exact sequence induces the following long exact sequence in group cohomology:
\[
	\begin{tikzcd}
		{}
			&{}
			&\cdots \ar[r]
			& H^{n-1}(G,A/rA)\\
		{} \ar[r]
			&H^n(G,A)\ar[r, "r"]
			&H^n(G,A)\ar[r]
			&H^{n}(G,A/rA)\\
		{} \ar[r]
			&H^{n+1}(G,A)\ar[r,"r"]
			& H^{n+1}(G,A) \ar[r]
			&\cdots .
	\end{tikzcd}
\]
Considering the image of
\[
	H^n(G,A/rA)\rightarrow H^{n+1}(G,A),
\]
we see that this image is the kernel of multiplication by $r$ on $H^{n+1}(G,A)$.
Hence we obtain an exact sequence
\[
	0\rightarrow H^n(G,A)/rH^n(G,A)\rightarrow H^n(G,A/rA)\rightarrow H^{n+1}(G,A)[r]\rightarrow 0.
\]
Here $H^n(G,A)[r]$ denotes the subgroup of elements of $H^n(G,A)$ killed by multiplication by $r$,
that is,
\[
	H^n(G,A)[r]:=\ker \left(r:H^n(G,A) \rightarrow H^n(G,A)\right).
\]
Similarly, for every $i\geqq 1$,
we write
\[
	H^n(G,A)[r^i]:=\ker\left(r^i:H^n(G,A) \rightarrow H^n(G,A)\right).
\]
Throughout the rest of this subsection,
we assume that multiplication by $r$ on $A$ is injective.
Then,
for every $i\geqq 1$,
multiplication by $r^i$ on $A$ is also injective.

Let $i<j$.
Define
\[
	\phi^1_{ij}:H^n(G,A)/r^jH^n(G,A)\rightarrow H^n(G,A)/r^iH^n(G,A)
\]
to be the natural map obtained by reducing modulo $r^i$ instead of modulo $r^j$.
Define
\[
	\phi^2_{ij}:H^n(G,A/r^jA)\rightarrow H^n(G,A/r^iA)
\]
to be the map induced by the natural reduction map on coefficients
\[
	A/r^jA\rightarrow A/r^i A~;~a+r^jA\mapsto a+r^iA.
\]
Finally,
define
\[
	\phi^3_{ij}:H^{n+1}(G,A)[r^j]\rightarrow H^{n+1}(G,A)[r^i]
\]
to be multiplication by $r^{j-i}$.
Then
\[
	0\rightarrow (H^n(G,A)/r^iH^n(G,A),\phi^1_{ij})\rightarrow  (H^n(G,A/r^iA),\phi^2_{ij})\rightarrow (H^{n+1}(G,A)[r^i],\phi^3_{ij})\rightarrow 0
\]
is a short exact sequence of inverse systems.

From the above discussion,
we obtain the following lemma.

\begin{lemma}
	For every $n\geqq 0$,
	there is an exact sequence
	\[
	0\rightarrow \varprojlim_{i\in \mathbb{N}}H^n(G,A)/r^iH^n(G,A)
	\rightarrow \varprojlim_{i\in \mathbb{N}}H^n(G,A/r^iA)
	\rightarrow \varprojlim_{i\in \mathbb{N}}H^{n+1}(G,A)[r^i]
	\rightarrow 0.
	\]
	Here the last inverse limit is taken with respect to the transition maps given by multiplication by powers of $r$.
\end{lemma}

\begin{proof}
	By left exactness of the inverse limit functor,
	it remains to show the right exactness.
	The inverse system
	\[
		\left(H^n(G,A)/r^iH^n(G,A),\phi^1_{ij}\right)
	\]
	has surjective transition maps,
	and hence satisfies the Mittag-Leffler condition.
	Therefore,
	by \cite[Part 1, Lemma 10.86.4]{stacks-project}
	the inverse limit of the above short exact sequence of inverse systems is exact.
\end{proof}

If $H^n(G,A)$ is Hausdorff with respect to the $r$-adic topology,
that is,
if
\[
	\bigcap_{i\geqq 1}r^iH^n(G,A)=\{ 0\},
\]
then there is an injection
\[
	H^n(G,A)\hookrightarrow \varprojlim_{i\in \mathbb{N}}H^n(G,A)/r^iH^n(G,A).
\]
Moreover,
if $H^n(G,A)$ is complete with respect to the $r$-adic topology,
then this injection is an isomorphism.

However,
in the situations considered later,
$H^n(G,A)$ may contain a nonzero $r$-divisible element.
In such cases,
$H^n(G,A)$ is not Hausdorff with respect to the $r$-adic topology.

\subsection{Group Cohomology with Coefficients in Discrete Torsion-Free Abelian Groups}

Throughout this subsection,
let $G$ be a profinite group,
and let all cohomology groups be continuous group cohomology with discrete coefficients.
We compute the cohomology of $G$ with coefficients in a torsion-free abelian group on which $G$ acts trivially.

\begin{lemma}
	Let $A$ be a discrete torsion-free abelian group on which $G$ acts trivially.
	Then the following statements hold:
	\begin{enumerate}
		\item $H^1(G,A)=0$.
		\item For every $n\geqq 1$,
		there is a natural isomorphism
		\[
			H^n(G,(A\otimes_{\mathbb{Z}}\mathbb{Q})/A)\cong H^{n+1}(G,A).
		\]
	\end{enumerate}
\end{lemma}

\begin{proof}
	\begin{enumerate}
		\item Since the action of $G$ on $A$ is trivial,
		we have
		\[
			H^1(G,A)=\operatorname{Hom}_{\mathrm{cts}}(G,A).
		\]
		Let $f\in \operatorname{Hom}_{\mathrm{cts}}(G,A)$.
		Since $G$ is compact and $A$ is discrete,
		the image $f(G)$ is a finite subgroup of $A$.
		Since $A$ is torsion-free,
		it has no non-trivial finite subgroup.
		Hence $f(G)=0$,
		and therefore $f=0$.
		Thus
		\[
			H^1(G,A)=0.
		\]
		\item Since $A$ is torsion-free,
		the natural map
		\[
			A\hookrightarrow A\otimes_\mathbb{Z}\mathbb{Q}
		\]
		is injective.
		Hence we have a short exact sequence of discrete $G$-modules
		\[
			0\rightarrow A\rightarrow A\otimes_\mathbb{Z}\mathbb{Q}\rightarrow (A\otimes_\mathbb{Z}\mathbb{Q})/A\rightarrow 0.
		\]
		Since $A\otimes_\mathbb{Z}\mathbb{Q}$ is a $\mathbb{Q}$-vector space,
		we have
		\[
			H^n(G,A\otimes_\mathbb{Z}\mathbb{Q})=0~(n\geqq 1).
		\]
		The long exact sequence in cohomology associated with the above short exact sequence therefore yields,
		for every $n\geqq 1$,
		an isomorphism
		\[
			H^n(G,(A\otimes_\mathbb{Z}\mathbb{Q})/A)\cong H^{n+1}(G,A).
		\]
		\end{enumerate}
\end{proof}

The quotient $(A\otimes_\mathbb{Z}\mathbb{Q})/A$ is a torsion abelian group.
Hence it can be written as a direct limit of finite abelian groups:
\[
	(A\otimes_\mathbb{Z}\mathbb{Q})/A\cong \varinjlim_{i\in I}F_i.
\]
We obtain the following consequence.

\begin{lemma}
	Let $A$ be a torsion-free abelian group,
	and suppose that
	\[
		(A\otimes_\mathbb{Z}\mathbb{Q})/A\cong \varinjlim_{i\in I}F_i
	\]
	is a filtered direct limit of finite abelian groups.
	Assume that $G$ acts trivially on each $F_i$.
	Then,
	for every $n\geqq 1$,
	there is a natural isomorphism
	\[
		H^{n+1}(G,A)\cong \varinjlim_{i\in I}H^n(G,F_i).
	\]
\end{lemma}

\begin{proof}
	By the preceding lemma,
	we obtain 
	\begin{align*}
		H^{n+1}(G,A)
		&\cong H^n(G,(A\otimes_\mathbb{Z}\mathbb{Q})/A)\\
		&\cong H^n(G,\varinjlim_{i\in I}F_i).
	\end{align*}
	Since continuous cohomology of profinite groups with discrete coefficients commutes with filtered direct limits in the coefficient module,
	we have
	\[
		H^n(G,\varinjlim_{i\in I}F_i)\cong \varinjlim_{i\in I}H^n(G,F_i).
	\]
	Combining these isomorphisms gives
	\[
		H^{n+1}(G,A)\cong \varinjlim_{i\in I}H^n(G,F_i).
	\]
\end{proof}

\subsection{Group Cohomology with Coefficients Equipped with the Profinite Topology}

Let $I$ be a countable directed set,
and let $A$ be the inverse limit of an inverse system $(A_i,\phi_{ij},I)$ of abelian groups.

The inverse limit functor is left exact.
Hence one can consider its right derived functors $\varprojlim^j$.
It is known that if the inverse system $(A_i)$ satisfies the Mittag-Leffler condition,
then
\[
	\varprojlim{}^j A_i=0
\]
for every $j\geqq 1$.

The following result is also known for cohomology:

\begin{lemma}\cite[Theorem 2.7.5]{cohomology-of-number-fields}
	Let $A$,
	as a $G$-module,
	be defined to be the inverse limit of finite discrete $G$-modules indexed by $\mathbb{N}$.
	Equip $A$ with the profinite topology.
	\begin{enumerate}
		\item For each $n\geqq 1$,
		there exists an exact sequence
		\[
		0\rightarrow {\varprojlim_{i}}^1H^{n-1}(G,A_i)\rightarrow H^n(G,A)\rightarrow \varprojlim_i H^n(G,A_i)\rightarrow 0.
		\]
		\item Suppose that $H^n(G,A_i)$ is finite for every $n\geqq 1$ and every $i$.
		Then,
		for every $n\geqq 1$,
		\[
			H^{n}(G,A)\cong \varprojlim_i H^{n}(G,A_i).
		\]
	\end{enumerate}
\end{lemma}

\subsection{The Relationship between Torsion-Free Discrete Coefficients and Profinite Coefficients}

Let $A$ be the inverse limit of an inverse system $(A_i)$ of finite abelian groups,
and suppose that $G$ acts trivially on $A$.
Let $A^p$ denote $A$ endowed with the profinite topology,
and let $A^d$ denote $A$ endowed with the discrete topology.
Then the identity map
\[
	A^d \rightarrow A^p
\]
is continuous.
Hence it induces a map on group cohomology
\[
	H^n(G,A^d)\rightarrow H^n(G,A^p).
\]
We study this map.

The following result shows that this map carries information only in the torsion part.

\begin{lemma}
	Assume that $A$ is torsion-free,
	and let $n\geqq 1$.
	The image of the map
	\[
		H^n(G,A^d)\rightarrow H^n(G,A^p)
	\]
	induced by the identity map $A^d\rightarrow A^p$ is contained in the torsion part,
	and this map is zero for $n=1$.
	In particular,
	if $H^n(G,A^p)$ is torsion-free for some $n\geqq 1$,
	this map is zero.
\end{lemma}

\begin{proof}
	For $n=1$,
	the assertion follows from $H^1(G,A^d)=0$.
	Since every element of $H^n(G,A^d)$ has finite order,
	the image is zero whenever $H^n(G,A^p)$ is torsion-free.
\end{proof}

Assume that $A$ is torsion-free,
and that $A\otimes_\mathbb{Z}\mathbb{Q}$ admits a topology making it a topological vector space over the discrete field $\mathbb{Q}$ in which $A^p$ is an open subgroup.
Denote the resulting space by $(A\otimes_{\mathbb{Z}}\mathbb{Q})^p$.
Then there is a short exact sequence of topological groups
\[
	0\rightarrow A^p\xrightarrow{\iota} (A\otimes_\mathbb{Z}\mathbb{Q})^p\xrightarrow{\pi} (A\otimes_\mathbb{Z}\mathbb{Q})^p/A^p\rightarrow 0.
\]
Here $\iota$ is given by $a\mapsto a\otimes_\mathbb{Z}1$,
and $\pi$ is reduction modulo $A$.
The natural quotient topology on
\[
	(A\otimes_\mathbb{Z}\mathbb{Q})^p/A^p
\]
is discrete.
Since every map from a discrete space to an arbitrary topological space is continuous,
a choice of representatives gives a continuous set-theoretic section
\[
	(A\otimes_\mathbb{Z}\mathbb{Q})^p/A^p\rightarrow (A\otimes_\mathbb{Z}\mathbb{Q})^p.
\]
Therefore,
by \cite{cohomology-of-number-fields}, Lemma 2.7.2,
we obtain the long exact sequence
\[
	\begin{tikzcd}
		{} &
			{} &
			\cdots \ar[r] &
			H^{n-1}(G,A\otimes_\mathbb{Z} \mathbb{Q}/A)\\
		{} \ar[r] &
			H^n(G,A^p) \ar[r] &
			H^n(G,(A\otimes_\mathbb{Z} \mathbb{Q})^p) \ar[r]&
			H^n(G,A\otimes_\mathbb{Z} \mathbb{Q}/A)\\
		{} \ar[r] &
			H^{n+1}(G,A^p) \ar[r] &
			\cdots &
			{}.
	\end{tikzcd}
\]
Here $(A\otimes_\mathbb{Z}\mathbb{Q})/A$ is endowed with the discrete quotient topology.
Since,
for $n\geqq 2$,
\[
	H^{n-1}(G,(A\otimes_{\mathbb{Z}}\mathbb{Q})/A)\cong H^n(G,A^d),
\]
we obtain the following exact sequence for $n\geqq 2$:
\[
\begin{tikzcd}
	{} &
	{} &
	\cdots \ar[r] &
	H^{n}(G,A^d)\\
	{} \ar[r] &
	H^n(G,A^p) \ar[r] &
	H^n(G,(A\otimes_\mathbb{Z} \mathbb{Q})^p) \ar[r]&
	H^{n+1}(G,A^d)\\
	{} \ar[r] &
	H^{n+1}(G,A^p) \ar[r] &
	\cdots &
	{}.
\end{tikzcd}
\]

We now consider the compatibility of group cohomology with tensoring with $\mathbb{Q}$.
For each $n\geqq 0$,
there is a natural comparison map
\[
	H^n(G,A^p)\otimes_\mathbb{Z}\mathbb{Q}\rightarrow H^n(G,(A\otimes_{\mathbb{Z}}\mathbb{Q})^p).
\]
The following proposition describes the image and kernel of the map from discrete to profinite coefficients under explicit hypotheses on these comparison maps.

\begin{proposition}\label{proposition1-1}
	Assume that $A$ is torsion-free.
	\begin{enumerate}
		\item Let $n\geqq 2$,
		and suppose that the natural comparison map
		\[
			H^n(G,A^p)\otimes_\mathbb{Z}\mathbb{Q}\cong H^n(G,(A\otimes_\mathbb{Z}\mathbb{Q})^p)
		\]
		is an isomorphism.
		Then the map
		\[
			H^n(G,A^d)\rightarrow H^n(G,A^p)
		\]
		induced by the identity map $A^d\rightarrow A^p$ is surjective onto $H^n(G,A^p)_\mathrm{tor}$.
		\item In addition to the hypothesis in $(1)$,
		suppose that the natural comparison map
		\[
			H^{n-1}(G,A^p)\otimes_\mathbb{Z} \mathbb{Q}\cong H^{n-1}(G,(A\otimes_\mathbb{Z} \mathbb{Q})^p)
		\]
		is an isomorphism.
		Then there is a split short exact sequence
		\[
			0\rightarrow H^{n-1}(G,A^p)\otimes_\mathbb{Z} \mathbb{Q}/\mathbb{Z}\rightarrow H^{n}(G,A^d)\rightarrow H^n(G,A^p)_{\mathrm{tor}}\rightarrow 0.
		\]
		In particular,
		\[
			H^n(G,A^d)\cong H^n(G,A^p)_{\mathrm{tor}} \oplus (H^{n-1}(G,A^p)\otimes_\mathbb{Z} \mathbb{Q}/\mathbb{Z}).
		\]
		\item Suppose that,
		for every $n\geqq 1$,
		the natural comparison map
		\[
			H^n(G,A^p)\otimes_\mathbb{Z}\mathbb{Q}\cong H^n(G,(A\otimes_\mathbb{Z}\mathbb{Q})^p)
		\]
		is an isomorphism,
		and suppose that $A$ is a profinite ring.
		For each $n\geqq 2$,
		identify
		\[
			Y^n:=H^{n-1}(G,A^p)\otimes_\mathbb{Z}\mathbb{Q}/\mathbb{Z}
		\]
		with the corresponding subgroup of $H^n(G,A^d)$ via the short exact sequence in $(2)$.
		Then,
		for every $m\geqq 1$,
		\[
			Y^n \cup H^m(G,A^d)=0,~H^m(G,A^d)\cup Y^n=0.
		\]
		Here the cup product is taken in $H^*(G,A^d)$.
	\end{enumerate}
\end{proposition}

\begin{proof}
	\begin{enumerate}
		\item Let $n\geqq 2$.
		The long exact sequence may be written as
\[
	\begin{tikzcd}
		{} &
		{} &
		\cdots \ar[r] &
		H^{n}(G,A^d)\\
	{} \ar[r] &
		H^n(G,A^p) \ar[r] &
		H^n(G,(A\otimes_\mathbb{Z} \mathbb{Q})^p) \ar[r]&
		H^{n+1}(G,A^d)\\
		{} \ar[r] &
		H^{n+1}(G,A^p) \ar[r] &
		\cdots &
		{}.
\end{tikzcd}
\]
		The kernel of
		\[
			H^n(G,A^p)\rightarrow H^n(G,A^p)\otimes_\mathbb{Z}\mathbb{Q}
		\]
		is precisely $H^n(G,A^p)_{\mathrm{tor}}$,
		and by exactness this kernel agrees with the image of $H^n(G,A^d)$.
		
		It remains to verify that the map appearing here is the map induced by the identity map
		\[
			f:A^d\rightarrow A^p.
		\]
		Consider the following morphism of short exact sequences:
		\[
			\begin{tikzcd}
			0 \ar[r] 
			& A^d \ar[r] \ar[d]
			& (A\otimes_\mathbb{Z}\mathbb{Q})^d\ar[r] \ar[d]
			& (A\otimes_\mathbb{Z}\mathbb{Q})/A \ar[r] \ar[d,"="]
			& 0\\
			0 \ar[r] 
			& A^p \ar[r]
			& (A\otimes_\mathbb{Z}\mathbb{Q})^p\ar[r]
			& (A\otimes_\mathbb{Z}\mathbb{Q})/A \ar[r]
			& 0.
		\end{tikzcd}
		\]
		The vertical maps are the identity maps on the underlying groups and are continuous.
		Naturality of the connecting homomorphisms gives a commutative diagram
		\[
			\begin{tikzcd}
				H^{n-1}(G,(A\otimes_\mathbb{Z}\mathbb{Q})/A) \ar[r, "\delta_d",] \ar[d,"="]
					& H^n(G,A^d) \ar[d, "H^n(f)"]\\
				H^{n-1}(G,(A\otimes_\mathbb{Z}\mathbb{Q})/A) \ar[r, "\delta_p"]
					& H^n(G,A^p) 
			\end{tikzcd}
		\]
		Here $\delta_d$ is an isomorphism,
		since
		\[
			H^i(G,(A\otimes_{\mathbb{Z}}\mathbb{Q})^d)=0~(i\geqq 1).
		\]
		Hence
		\[
			H^n(f)=\delta_p\circ \delta_d^{-1},
		\]
		and the assertion follows.
		\item Consider the short exact sequence
		\[
			0\rightarrow \ker H^n(f)\rightarrow H^n(G,A^d) \rightarrow H^n(G,A^p)_{\mathrm{tor}}\rightarrow 0.
		\]
		By exactness,
		\[
			\ker H^n(f)\cong \operatorname{coker}\left(H^{n-1}(G,A^p)\rightarrow H^{n-1}(G,(A\otimes_\mathbb{Z}\mathbb{Q})^p)\right).
		\]
		Under the comparison isomorphism,
		the displayed map corresponds to the natural map
		\[
			H^{n-1}(G,A^p)\rightarrow H^{n-1}(G,A^p)\otimes_{\mathbb{Z}}\mathbb{Q}.
		\]
		Therefore
		\begin{align*}
			\ker H^n(f)
				&\cong \operatorname{coker} \left( H^{n-1}(G,A^p)\rightarrow H^{n-1}(G,(A\otimes_\mathbb{Z}\mathbb{Q})^p)\right)\\
				&\cong H^{n-1}(G,A^p)\otimes_\mathbb{Z} (\mathbb{Q}/\mathbb{Z})
		\end{align*}
		Thus the asserted short exact sequence follows.
		The left-side term is divisible,
		hence injective as a $\mathbb{Z}$-module,
		so the sequence splits.
		\item Let $n\geqq 2$,
		and let $m\geqq 1$.
		If $m=1$,
		then $H^1(G,A^d)=0$,
		so the assertion is immediate.
		Assume that $m\geqq 2$.
		Set
		\[
			X^m=H^m(G,A^p)_\mathrm{tor}.
		\]
		By $(2)$,
		\[
			H^m(G,A^d)\cong X^m\oplus Y^m.
		\]
		The cup product factors through the tensor product.
		Since $Y^n$ and $Y^m$ are divisible torsion groups and $X^m$ is a torsion group,
		\[
			Y^n\otimes_\mathbb{Z} X^m=0,~Y^n\otimes_\mathbb{Z} Y^m=0.
		\]
		Hence
		\[
			Y^n\cup H^m(G,A^d)=0.
		\]
		The same argument immediately gives
		\[
			H^m(G,A^d)\cup Y^n=0.
		\]
	\end{enumerate}
\end{proof}

We now consider the natural map
\[
	H^*(G,A^d)\rightarrow H^*(G,A^p).
\]
The natural projections $A\rightarrow A_i$ induce a compatible system of cochain maps
\[
	C^\bullet (G,A^d)\rightarrow C^\bullet(G,A_i).
\]
By the universal property of the inverse limit,
these maps determine a cochain map
\[
	C^\bullet (G,A^d)\rightarrow \varprojlim_{i\in I} C^\bullet(G,A_i)\cong C^\bullet (G,A^p).
\]
Passing to cohomology gives a map
\[
	H^*(G,A^d)\rightarrow H^*(G,A^p).
\]
We then have the following proposition.

\begin{proposition}\label{proposition1-2}
	The map
	\[
		H^*(G,A^d)\rightarrow H^*(G,A^p)
	\]
	induced by the identity map $A^d\rightarrow A^p$ agrees with the map obtained by applying the universal property of the inverse limit to the continuous cochain complexes and then passing to cohomology.
\end{proposition}

\begin{proof}
	Let
	\[
		p_i:A^p\rightarrow A_i
	\]
	be the projection,
	and let
	\[
		\iota:A^d\rightarrow A^p
	\]
	be the identity map.
	Then the natural projection
	\[
		A^d\rightarrow A_i
	\]
	is equal to the composite $p_i\circ \iota$.
	
	By the proof of Theorem 2.7.5 of \cite{cohomology-of-number-fields},
	the canonical map
	\[
		\phi:C^\bullet (G,A^p)\cong \varprojlim_{i\in I}C^\bullet(G,A_i)
	\]
	induced by the maps
	\[
		C^\bullet(G,A^p)\rightarrow C^\bullet(G,A_i)
	\]
	is an isomorphism of cochain complexes.
	Moreover,
	$\phi$ is characterized by the property that,
	for every $i\in I$,
	\[
		\operatorname{pr}_i\circ \phi=C^\bullet(p_i),
	\]
	where
	\[
		\operatorname{pr}_i:\varprojlim_{i\in I}C^\bullet(G,A_i)\rightarrow C^\bullet(G,A_i)
	\]
	denotes the canonical projection.
	
	Now let
	\[
		\psi:C^\bullet(G,A^d)\rightarrow \varprojlim_{i\in I}C^\bullet(G,A_i)
	\]
	be the cochain map obtained from the universal property of the inverse limit,
	applied to the compatible family of cochain maps
	\[
		C^\bullet(p_i\circ \iota):C^\bullet(G,A^d)\rightarrow C^\bullet(G,A_i).
	\]
	For every $i\in I$,
	we have
	\begin{align*}
		\operatorname{pr}_i\circ \phi\circ C^\bullet(\iota)
			&=C^\bullet(p_i)\circ C^\bullet(\iota)\\
			&=C^\bullet(p_i\circ \iota).
	\end{align*}
	On the other hand,
	by the defining property of $\psi$,
	\[
		\operatorname{pr}_i\circ \psi=C^\bullet(p_i\circ \iota).
	\]
	Therefore
	\[
		\operatorname{pr}_i\circ \phi \circ C^\bullet(\iota)=\operatorname{pr}_i\circ\psi.
	\]
	By the uniqueness part of the universal property of the inverse limit,
	it follows that
	\[
		\phi\circ C^\bullet(\iota)=\psi.
	\]
	Hence 
	\[
		C^\bullet(\iota)=\phi^{-1}\circ \psi.
	\]
	Passing to cohomology,
	the map
	\[
		H^*(G,A^d)\rightarrow H^*(G,A^p)
	\]
	induced by the identity map $\iota:A^d\rightarrow A^p$ agrees with the map obtained from the universal property of the inverse limit.
\end{proof}

We denote the natural map above by
\[
	H(\iota):H^*(G,A^d)\rightarrow H^*(G,A^p).
\]
If $A$ is a profinite ring with its given topology,
then this map is compatible with cup products and hence is a homomorphism of graded rings.
For every degree $n\geqq 1$,
the image of
\[
	H(\iota)^n:H^n(G,A^d)\rightarrow H^n(G,A^p)
\]
is contained in the torsion subgroup $H^n(G,A^p)_{\operatorname{tor}}$.

Since $G$ acts trivially on $A^p$,
we have
\[
	H^1(G,A^p)=\operatorname{Hom}_{\mathrm{cont}}(G,A^p).
\]
This group is torsion-free,
and hence
\[
	H^1(G,A^p)_\mathrm{tor}=0.
\]
Keeping this in mind,
degreewise isomorphisms and the information about cup products combine to give the following description of the cohomology ring.

\begin{corollary}\label{theorem2}
	Assume that $A$ has a structure of a profinite ring compatible with the structures introduced above.
	Assume also that,
	for every $n\geqq 1$,
	the natural comparison map
	\[
		H^n(G,A^p)\otimes_\mathbb{Z}\mathbb{Q}\cong H^n(G,(A\otimes_\mathbb{Z}\mathbb{Q})^p)
	\]
	is an isomorphism.
	For every $n\geqq 0$,
	define
	\begin{align*}
		X^n=
		\begin{cases}
			H^0(G,A^p)\cong A & (n=0)\\
			H^n(G,A^p)_\mathrm{tor} & (n\geqq 1),
		\end{cases}\\
		Y^n=
		\begin{cases}
			0 & (n=0,1)\\
			H^{n-1}(G,A^p)\otimes_\mathbb{Z}(\mathbb{Q}/\mathbb{Z}) & (n\geqq 2).
		\end{cases}
	\end{align*}
	Let 
	\[
		\overline{H(\iota)}:H^*(G,A^d)\rightarrow X^*,~\iota:Y^*\rightarrow H^*(G,A^d)
	\]
	be the natural maps.
	Here $\overline{H(\iota)}$ is the following graded ring homomorphism: in degree $0$,
	it is the identity map,
	and in every degree $n\geqq 1$,
	it is $H(\iota)^n$ regarded as a map whose image lies in the torsion subgroup
	\[
		H^n(G,A^p)_\mathrm{tor}=X^n.
	\]
	For each $n\geqq 0$,
	choose and fix an additive section
	\[
		s^n:X^n\rightarrow H^n(G,A^d)
	\]
	of $\overline{H(\iota)}$,
	and put
	\[
		s=\bigoplus_{n\geqq 0}s^n:X^*\rightarrow H^*(G,A^d).
	\]
	For $n=0$,
	we take $s^0$ to be the identity map,
	and for $n=1$,
	we take the zero map
	\[
		X^1=0\rightarrow H^1(G,A^d).
	\]
	Define
	\[
		\phi_s:X^*\oplus Y^*\rightarrow H^*(G,A^d)~;~\phi_s(x,y)=s(x)+\iota(y).
	\]
	Then $\phi_s$ is the isomorphism of abelian groups obtained in Proposition \ref{proposition1-1}.
	Define a product $*$ on $X^*\oplus Y^*$ by
	\begin{align*}
		&(x_1,y_1)*(x_2,y_2)\\
			&=(x_1\cup x_2,\operatorname{pr}_Y\circ \phi_s^{-1}\left( s(x_1)\cup \iota(y_2)+\iota(y_1)\cup s(x_2)+s(x_1)\cup s(x_2)-s(x_1\cup x_2)\right)),
	\end{align*}
	where $\operatorname{pr}_Y$ denotes the projection onto the $Y^*$-component.
	Then
	\[
		(X^*\oplus Y^*)\cong H^*(G,A^d)
	\]
	as graded rings.
	In particular,
	$H^*(G,A^d)$ is a square-zero extension of $X^*$ by $Y^*$.
\end{corollary}

\begin{proof}
	By Proposition \ref{proposition1-1} (2),
	$\phi_s$ is an isomorphism of abelian groups.
	Let $\alpha,\beta\in X^*\oplus Y^*$,
	and write
	\[
		\alpha=x_1+y_1,~\beta=x_2+y_2.
	\]
	We define a product $*$ on $X^*\oplus Y^*$ via $\phi_s$ as follows:
	\[
		\alpha*\beta=\phi_s^{-1}(\phi_s(\alpha)\cup \phi_s(\beta)),
	\]
	where $\cup$ denotes the cup product on $H^*(G,A^d)$.
	It remains to compute this product explicitly.
	
	Since
	\[
		\phi_s(\alpha)=s(x_1)+\iota(y_1),~\phi_s(\beta)=s(x_2)+\iota(y_2),
	\]
	we have
	\begin{align*}
		\phi_s(\alpha)\cup \phi_s(\beta)
			&=(s(x_1)+\iota(y_1))\cup (s(x_2)+\iota(y_2))\\
			&=s(x_1)\cup s(x_2)+s(x_1)\cup \iota(y_2)+\iota(y_1)\cup s(x_2)+\iota(y_1)\cup \iota(y_2).
	\end{align*}
	Since $Y^0=Y^1=0$ and $Y^n~(n\geqq 2)$ is a divisible torsion group,
	the image of $Y^*$ under $\iota$ has the same property.
	Moreover,
	cup products factor through tensor products,
	and the tensor product of two divisible torsion groups is zero.
	Therefore
	\[
		\iota(y_1) \cup \iota(y_2)=0.
	\]
	Next,
	since $\overline{H(\iota)}$ is a homomorphism of graded rings and 
	\[
		\iota(Y^*)=\ker(\overline{H(\iota)}),
	\]
	we have
	\[
		\overline{H(\iota)}(s(x_1)\cup \iota(y_2))
			=\overline{H(\iota)}(s(x_1))\cup \overline{H(\iota)}(\iota(y_2))
			=x_1\cup 0
			=0.
	\]
	Hence 
	\[
		s(x_1)\cup \iota(y_2)\in \iota(Y^*).
	\]
	Similarly,
	\[
		\iota(y_1)\cup s(x_2)\in \iota(Y^*).
	\]
	Finally,
	we consider the term $s(x_1)\cup s(x_2)$.
	Since $\overline{H(\iota)}$ is a homomorphism of graded rings and $\overline{H(\iota)}\circ s=\mathrm{id}_{X^*}$,
	we have
	\[
		\overline{H(\iota)}(s(x_1)\cup s(x_2))=x_1\cup x_2.
	\]
	Hence the $X^*$-component of
	\[
		\phi^{-1}_s(s(x_1)\cup s(x_2))
	\]
	is $x_1\cup x_2$.
	
	The corresponding $Y^*$-component is obtained by subtracting the chosen lift $s(x_1\cup x_2)$ of this $X^*$-component.
	Indeed,
	\[
		\overline{H(\iota)}(s(x_1)\cup s(x_2)-s(x_1\cup x_2))=x_1\cup x_2-x_1\cup x_2=0,
	\]
	and hence
	\[
		s(x_1)\cup s(x_2)-s(x_1\cup x_2)\in \iota(Y^*).
	\]
	Therefore the $Y^*$-component of
	\[
		\phi^{-1}_s(s(x_1)\cup s(x_2))
	\]
	is
	\[
		\operatorname{pr}_Y\circ \phi^{-1}_s(s(x_1)\cup s(x_2)-s(x_1\cup x_2)).
	\]
	Combining this with the preceding discussion of the terms $s(x_1)\cup \iota(y_2)$ and $\iota(y_1)\cup s(x_2)$,
	we see that the $X^*$-component of $\alpha*\beta$ is $x_1\cup x_2$,
	while its $Y^*$-component is
	\[
		\operatorname{pr}_Y\circ \phi^{-1}_s(s(x_1)\cup \iota(y_2)+\iota(y_1)\cup s(x_2)+s(x_1)\cup s(x_2)-s(x_1\cup x_2)).
	\]
	Thus the product transported by $\phi_s$ is exactly the product $*$ defined above.
	Consequently,
	$\phi_s$ is an isomorphism of graded rings.
\end{proof}

% !TeX root = main.tex
\section{The Structure of \texorpdfstring{$\operatorname{Ext}(R[G/K],-)$}{Ext(R[G/K],-)} with Coefficients in a Torsion-Free Profinite Ring}

\subsection{The Structure of the Derived Hecke Algebra as an Abelian Group}
Let $G$ be a locally profinite group,
let $K$ be an open compact subgroup of $G$,
and let $R$ be a torsion-free commutative ring.
Suppose that $G$ acts trivially on $R$.
Assume that there is a filtered direct system $(R_i)_{i\in I}$ of finite discrete abelian groups such that we have an isomorphism of abelian groups
\[
(R\otimes_\mathbb{Z}\mathbb{Q})/R\cong \varinjlim_{i\in I}R_i. 
\]
Then the following holds for the derived Hecke algebra $\mathscr{H}(G,K)_R$.

\begin{proposition}\label{proposition3}
	Suppose that $R$ is torsion-free and that $G$ acts trivially on $R$.
	Let $n$ be a nonnegative integer.
	Then we have the following isomorphisms as abelian groups:
	\begin{enumerate}
		\item If $n=0$,
		then
		\[
		\mathscr{H}^0(G,K)_R\cong \bigoplus_{x\in K\backslash G/K}R.
		\]
		\item If $n=1$,
		then
		\[
		\mathscr{H}^1(G,K)_R=0.
		\]
		\item If $n\geqq 2$ and each $R_i$ has the structure of a finite commutative ring,
		then
		\[
		\mathscr{H}^n(G,K)_R\cong \varinjlim_{i\in I}\mathscr{H}^{n-1}(G,K)_{R_i}.
		\]
	\end{enumerate}
\end{proposition}

\begin{proof}
	For any $n$,
	we have
	\[
	\mathscr{H}^n(G,K)_R\cong \bigoplus_{x\in K\backslash G/K}H^n(K_x,R).
	\]
	Hence,
	if $n=0$,
	then $H^0(K_x,R)=R$.
	If $n=1$,
	then since $R$ is torsion-free,
	we have $H^1(K_x,R)=0$.
	If $n\geqq 2$ and each $R_i$ has the structure of a finite commutative ring,
	then for each $x\in K\backslash G/K$ we have
	\[
	H^n(K_x,R)\cong \varinjlim_{i\in I}H^{n-1}(K_x,R_i).
	\]
	Therefore,
	noting that direct limits commute with direct sums,
	we obtain
	\begin{align*}
		\mathscr{H}^n(G,K)_R
		&\cong \bigoplus_{x\in K\backslash G/K}H^n(K_x,R)\\
		&\cong \bigoplus_{x\in K\backslash G/K}\varinjlim_{i\in I}H^{n-1}(K_x,R_i)\\
		&\cong \varinjlim_{i\in I}\bigoplus_{x\in K\backslash G/K}H^{n-1}(K_x,R_i)\\
		&\cong \varinjlim_{i\in I}\mathscr{H}^{n-1}(G,K)_{R_i}.
	\end{align*}
\end{proof}

We also obtain the following description from group cohomology with coefficients endowed with the profinite topology.

\begin{proposition}\label{proposition4}
	Let $R$ be a torsion-free profinite ring,
	and let $R^p$ denote $R$ endowed with the profinite topology.
	Suppose moreover that,
	for any $n\geqq 1$ and any $x\in K\backslash G/K$,
	the natural map
	\[
		H^n(K_x,R^p)\otimes_\mathbb{Z}\mathbb{Q}
			\rightarrow H^n(K_x,R^p\otimes_\mathbb{Z}\mathbb{Q})
	\]
	is an isomorphism.
	Then,
	for any $m\geqq 2$,
	there exists an isomorphism of abelian groups
	\[
		\mathscr{H}^m(G,K)_R
			\cong	\bigoplus_{x\in K\backslash G/K}
			\left(
				(H^{m-1}(K_x,R^p)\otimes_\mathbb{Z} (\mathbb{Q}/\mathbb{Z}))\oplus H^m(K_x,R^p)_{\mathrm{tor}}
			\right).
	\]
	Moreover,
	using the notation of Corollary \ref{theorem2},
	for $x\in K\backslash G/K$,
	let $X^\bullet_x$ and $Y^\bullet_x$ be the corresponding subgroup of $H^*(K_x,R^p)$ and the object obtained by tensoring,
	respectively.
	Then,
	as graded abelian groups,
	\[
		\mathscr{H}(G,K)_R
			\cong \bigoplus_{x\in K\backslash G/K}
			\left(
				X^\bullet_x\oplus Y^\bullet_x
			\right).
	\]
\end{proposition}

\begin{proof}
	Considering the description by group cohomology
	\[
	\mathscr{H}^*(G,K)_R\cong \bigoplus_{x\in K\backslash G/K}H^*(K_x,R^d),
	\]
	by Corollary \ref{theorem2},
	for $m\geqq 2$ and for each $x$,
	we have
	\[
	H^m(K_x,R^d)\cong (H^{m-1}(K_x,R^p)\otimes_{\mathbb{Z}}(\mathbb{Q}/\mathbb{Z}))\oplus H^m(K_x,R^p)_{\mathrm{tor}}.
	\]
	Moreover,
	if $m=1$,
	then for each $x\in K\backslash G/K$ we have
	\[
	H^1(K_x,R^d)=0
	\]
	and $X^1_x=Y^1_x=0$.
	If $m=0$,
	then
	\[
	X^0_x=H^0(K_x,R^d)=R^d,~Y^0_x=0.
	\]
	Hence we obtain the desired isomorphism.
\end{proof}

\subsection{The action of the derived Hecke algebra with coefficients in a torsion-free profinite ring}

We next consider the action of the derived Hecke algebra with coefficients in a torsion-free profinite ring.
Let $L^\bullet$ be a complex in $\mathrm{SM}_R(G)$.
Let $I^\bullet$ be a K-injective resolution of $R[G/K]$,
and let $I_L^\bullet$ be a K-injective resolution of $L^\bullet$.
We define the action of $\operatorname{Ext}^*_{\mathrm{SM}_R(G)}(R[G/K],R[G/K])$ on $\operatorname{Ext}^*_{\mathrm{SM}_R(G)}(R[G/K],L^\bullet)$ to be the action induced by composition of morphisms:
\begin{align*}
	\operatorname{Hom}^*_{\mathrm{SM}_R(G)}(I^\bullet,I_L^\bullet)\times \operatorname{Hom}^*_{\mathrm{SM}_R(G)}(I^\bullet,I^\bullet)
	&\rightarrow \operatorname{Hom}^*_{\mathrm{SM}_R(G)}(I^\bullet,I_L^\bullet)\\
	(f,g)
	&\mapsto f\circ g.
\end{align*}
By this definition,
the action on Ext is bilinear,
and hence factors through the tensor product.
Namely,
the following diagram commutes: 
\[
\begin{tikzcd}
	\operatorname{Ext}^*_{\mathrm{SM}_R(G)}(R[G/K],L^\bullet) \times \mathscr{H}(G,K)_R \ar[r] \ar[d] &
	\operatorname{Ext}^*_{\mathrm{SM}_R(G)}(R[G/K],L^\bullet)\\
	\operatorname{Ext}^*_{\mathrm{SM}_R(G)}(R[G/K],L^\bullet)\otimes_{\mathbb{Z}} \mathscr{H}(G,K)_R \ar[ur] &
	{}
\end{tikzcd}.
\]
By the preceding discussion,
the divisible part acts trivially on the torsion part of
\[
\operatorname{Ext}^*_{\mathrm{SM}_R(G)}(R[G/K],L^\bullet).
\]

Therefore,
in order to understand how much the action simplifies,
we study the structure of 
\[
\operatorname{Ext}^*_{\mathrm{SM}_R(G)}(R[G/K],L^\bullet),
\]
especially its torsion property.
For simplicity,
we assume that $L^\bullet$ is bounded below;
more precisely,
we assume that
\[
L^n=0~(n<0).
\]

We first consider a description by means of a filtered colimit.
By Shapiro's lemma,
\[
\operatorname{Ext}^*_{\mathrm{SM}_R(G)}(R[G/K],L^\bullet)\cong \operatorname{Ext}^*_{\mathrm{SM}_R(K)}(R,L^\bullet)
\]
and this is equal to the group hypercohomology
\[
\mathbb{H}(K,L^\bullet).
\]
Since each term of $L^\bullet$ is a smooth representation,
we may write
\[
L^\bullet=\bigcup_{U}(L^\bullet)^U=\varinjlim_U (L^\bullet)^U.
\]
Here $U$ runs over all open normal subgroups of $K$.

Similarly,
for the continuous cochain complex $C^*_{\operatorname{cts}}(K,-)$,
for every $p,q$ we have
\[
C_{\operatorname{cts}}^p(K,L^q)\cong \varinjlim_U C^p(K/U,(L^q)^U).
\]
Since filtered colimits are exact,
we obtain
\begin{align*}
	\operatorname{Ext}^n_{\mathrm{SM}_R(K)}(R,L^\bullet)
	&\cong \mathbb{H}^n(K,L^\bullet)\\
	&\cong \varinjlim_U \mathbb{H}^n(K/U,(L^\bullet)^U).
\end{align*}

We next compute using a spectral sequence.
The group hypercohomology $\mathbb{H}(K,L^\bullet)$ is expressed by the spectral sequence
\begin{align}
	E_2^{p,q}=H^p(K,h^q(L^\bullet))\Rightarrow \mathbb{H}^{p+q}(K,L^\bullet),
\end{align}
where $h(L^\bullet)$ denotes the cohomology of $L^\bullet$.

In a more general situation,
considering this spectral sequence gives the following exact sequence.
It also shows that,
in order to study torsion properties,
it is enough to understand the cohomology of $L^\bullet$.

\begin{proposition}
	Let $R$ be an arbitrary commutative ring and let $K$ be a profinite group.
	Let $L^\bullet$ be a complex of smooth $R[K]$-modules satisfying
	\[
	L^n=0~(n<0),
	\]
	and let $h(L^\bullet)$ denote the cohomology of $L^\bullet$.
	Then,
	for every $n\geqq 0$,
	there is a short exact sequence
	\[
	0
	\rightarrow F^1 \mathbb{H}^n(K,L^\bullet)
	\rightarrow \mathbb{H}^n(K,L^\bullet)
	\rightarrow E_\infty^{0,n}
	\rightarrow 0.
	\]
	Here $F^\bullet \mathbb{H}^n(K,L^\bullet)$ denotes the filtration obtained from the spectral sequence (1).
	Then $F^1 \mathbb{H}^n(K,L^\bullet)$ is a torsion group.
	Moreover we have $E_\infty^{0,n}\subset E_2^{0,n}=h^n(L^\bullet)^K$.
	In particular,
	if $E_\infty^{0,n}$ is a torsion group for some $n\geqq 0$,
	then $\mathbb{H}^n(K,L^\bullet)$ is a torsion group.
\end{proposition}

\begin{proof}
	By assumption,
	the spectral sequence (1)
	is concentrated in the first quadrant $(p,q\geqq 0)$.
	Therefore,
	for each $n$,
	we obtain a finite filtration
	\[
	\mathbb{H}^{n}(K,L^\bullet)=F^0\mathbb{H}^n(K,L^\bullet)\supset F^1\mathbb{H}^n(K,L^\bullet)\supset \cdots \supset F^{n+1}\mathbb{H}^n(K,L^\bullet)=0,
	\]
	such that,
	for every $p,q$,
	\[
	\operatorname{gr}_p\mathbb{H}^{p+q}(K,L^\bullet)=F^p\mathbb{H}^{p+q}(K,L^\bullet)/F^{p+1}\mathbb{H}^{p+q}(K,L^\bullet)\cong E_\infty^{p,q}.
	\]
	If $p>0$,
	then since $K$ is profinite,
	each $E_2^{p,q}$ is a torsion group.
	The term $E_\infty^{p,q}$ is obtained from $E_2^{p,q}$ by repeatedly taking subquotients.
	Since the torsion property is preserved under subquotients,
	$E_\infty^{p,q}$ is also torsion for $p>0$.
	
	When $n=0$,
	the spectral sequence is concentrated in the first quadrant,
	so
	\[
	F^1 \mathbb{H}^0(K,L^\bullet)=0.
	\]
	Thus the assertion is clear in this case.
	Now let $n>0$.
	For $\mathbb{H}^n(K,L^\bullet)$,
	we have
	\begin{align*}
		E^{0,n}_\infty&\cong  F^0\mathbb{H}^n(K,L^\bullet)/F^1\mathbb{H}^n(K,L^\bullet)\\
		E^{1,n-1}_\infty&\cong F^1\mathbb{H}^n(K,L^\bullet)/F^2\mathbb{H}^n(K,L^\bullet)\\
		\vdots\\
		E^{n,0}_\infty&\cong F^{n}\mathbb{H}^n(K,L^\bullet)/F^{n+1}\mathbb{H}^n(K,L^\bullet)=F^n\mathbb{H}^n(K,L^\bullet).
	\end{align*}
	Since $E_\infty^{n,0}$ is torsion,
	$F^n \mathbb{H}^n(K,L^\bullet)$ is torsion.
	For $n\geqq 2$,
	we also have
	\[
	E^{n-1,1}_\infty\cong F^{n-1}\mathbb{H}^n(K,L^\bullet)/F^n\mathbb{H}^n(K,L^\bullet)
	\]
	and both $E_\infty^{n-1,1}$ and $F^n \mathbb{H}^n(K,L^\bullet)$ are torsion.
	Thus $F^{n-1} \mathbb{H}^n(K,L^\bullet)$ is torsion.
	Repeating this argument,
	we conclude that 
	\[
	F^1 \mathbb{H}^n(K,L^\bullet)
	\]
	is torsion.
	
	For $p=0$,
	we have
	\[
	E^{0,n}_\infty\cong F^0\mathbb{H}^n(K,L^\bullet)/F^1\mathbb{H}^n(K,L^\bullet).
	\]
	Hence we obtain the short exact sequence
	\[
	0	\rightarrow F^1 \mathbb{H}^n(K,L^\bullet)
	\rightarrow \mathbb{H}^n(K,L^\bullet)
	\rightarrow 	E^{0,n}_\infty 
	\rightarrow 0.
	\]
	If $E_\infty^{0,n}$ is torsion,
	then both ends of this short exact sequence are torsion groups.
	Therefore the middle term is also torsion.
\end{proof}

Thus,
for elements of $\mathbb{H}^q(K,L^\bullet)$,
the action of the divisible part factors through
\[
\mathbb{H}^q(K,L^\bullet)/F^1 \mathbb{H}^q(K,L^\bullet)\cong E^{0,q}_\infty.
\]
Combining these observations,
we obtain the following.

\begin{corollary}
	If $h^i(L^\bullet)^K$ is torsion,
	then the divisible part of
	\[
	\operatorname{Ext}^*_{\operatorname{SM}_R(G)}(R[G/K],R[G/K])
	\]
	acts trivially on
	\[
	\operatorname{Ext}^i_{\operatorname{SM}_R(G)}(R[G/K],L^\bullet).
	\]
\end{corollary}

Finally,
we consider the action of a divisible element on an element which is neither torsion nor divisible.

\begin{lemma}
	Let $L^\bullet$ be a complex in $\operatorname{SM}_R(G)$.
	Then,
	for every divisible element
	\[
	d\in \mathscr{H}(G,K)_R
	\]
	and every
	\[
	f\in \operatorname{Ext}^*_{\operatorname{SM}_R(G)}(R[G/K],L^\bullet),
	\]
	the element $f\circ d$ is divisible.
\end{lemma}

\begin{proof}
	Since $d$ is divisible,
	for every $n>0$ there exists an element
	\[
	d_n \in \mathscr{H}(G,K)_R
	\]
	such that 
	\[
	nd_n=d.
	\]
	Therefore,
	considering $f\circ d$,
	we have for every $n>0$
	\[
	f\circ d=f\circ (nd_n)=n(f\circ d_n).
	\]
	Thus $f\circ d$ is divisible.
\end{proof}

Therefore the action of the divisible part always has divisible image.
In general,
the divisible part may be very large.
Hence,
if some finiteness condition is imposed on
\[
\operatorname{Ext}^*_{\operatorname{SM}_R(G)}(R[G/K],L^\bullet),
\]
the action of the divisible part may become trivial.

% !TeX root = main.tex
\section{Multiplication in the Derived Hecke Algebra with Coefficients in a Profinite Ring}
Using the preceding discussion,
we now study the derived Hecke algebra.

Let $G$ be a locally profinite group,
and let $K$ be an open compact subgroup of $G$.
Let $(R_i)$ be an inverse system of finite rings,
and let $R$ be a torsion-free commutative ring such that
\[
	R\cong\varprojlim_i R_i.
\]
Moreover,
for each $x\in K\backslash G/K$,
set
\[
	K_x=K\cap g_xKg_x^{-1}.
\]

\subsection{Inverse limits of derived Hecke algebras}
All inverse limits of graded modules and graded rings in this section are taken degreewise.
For each $i\in I$ and $x\in K\backslash G/K$,
let
\[
	\phi^x_i:H^*(K_x,R^d)\rightarrow H^*(K_x,R_i)
\]
denote the map on group cohomology induced by the coefficient map $R^d\rightarrow R_i$.
Since the family
\[
	(\phi^x_i)_{i\in I}
\]
is compatible with the transition maps
\[
	H^*(K_x,R_i)\rightarrow H^*(K_x,R_j)
\]
induced by the morphisms $R_i\rightarrow R_j$ in the inverse system,
the universal property of the inverse limit gives a natural map
\[
	H^*(K_x,R^d)\rightarrow \varprojlim_{i\in I}H^*(K_x,R_i).
\]
Therefore,
by taking the direct sum of these maps,
we obtain a map from the derived Hecke algebra
\[
	\mathscr{H}(G,K)_R\rightarrow \bigoplus_{x\in K\backslash G/K} \varprojlim_{i\in I}H^*(K_x,R_i).
\]
We set
\begin{align*}
	\overline{\mathscr{H}}^n(G,K)_R
		&:=\bigoplus_{x\in K\backslash G/K} \varprojlim_{i\in I}H^n(K_x,R_i),\\
	\overline{\mathscr{H}}(G,K)_R
		&:=\bigoplus_{n\geqq 0}\overline{\mathscr{H}}^n(G,K)_R.
\end{align*}
Moreover,
the maps
\[
	H^*(K_x,R_i)\rightarrow H^*(K_x,R_j)
\]
induced by the morphisms $R_i\rightarrow R_j$ in the inverse system give rise to maps between derived Hecke algebras
\[
	\mathscr{H}(G,K)_{R_i}\rightarrow \mathscr{H}(G,K)_{R_j}.
\]
We may therefore consider the inverse limit
\begin{align*}
	\widehat{\mathscr{H}}^n(G,K)_R
		&:=\varprojlim_{i\in I} \mathscr{H}^n(G,K)_{R_i}\cong \varprojlim_{i\in I}\bigoplus_{x\in K\backslash G/K}H^n(K_x,R_i),\\
	\widehat{\mathscr{H}}(G,K)_R
		&:=\bigoplus_{n\geqq 0}\widehat{\mathscr{H}}^n(G,K)_R.
\end{align*}
The products on $\mathscr{H}(G,K)_{R_i}$ induce a graded multiplication on $\widehat{\mathscr{H}}(G,K)_R$.

Every element
\[
	a\in \overline{\mathscr{H}}(G,K)_R
\]
can be expressed as a family
\[
	(a_x)_{x\in K\backslash G/K},
\]
where
\[
	a_x\in \varprojlim_{i\in I}H^*(K_x,R_i)
\]
and $a_x=0$ for all but finitely many $x\in K\backslash G/K$.
Moreover,
each $a_x$ can be expressed as a compatible family
\[
	a_x=(a_{x,i})_{i\in I},~a_{x,i}\in H^*(K_x,R_i),
\]
with respect to the transition maps of the inverse system.
Thus,
$a$ can be written as 
\[
	a=((a_{x,i})_{i\in I})_{x\in K\backslash G/K}.
\]
In general,
the assignment
\[
	a=((a_{x,i})_{i\in I})_{x\in K\backslash G/K}\mapsto ((a_{x,i})_{x\in K\backslash G/K})_{i\in I}
\]
defines a natural map
\[
	\overline{\mathscr{H}}(G,K)_R\rightarrow \widehat{\mathscr{H}}(G,K)_R,
\]
and this map is injective.

We now consider the image of $\mathscr{H}(G,K)_R$.
For
\[
	a=(a_x)_{x\in K\backslash G/K}\in \mathscr{H}(G,K)_R,
\]
its image in $\overline{\mathscr{H}}(G,K)_R$ is
\[
	((\phi^x_i(a_x))_{i\in I})_{x\in K\backslash G/K},
\]
whereas its image in $\widehat{\mathscr{H}}(G,K)_R$ is
\[
	((\phi^x_i(a_x))_{x\in K\backslash G/K})_{i\in I}.
\]

The image of $\overline{\mathscr{H}}(G,K)_R$ in $\widehat{\mathscr{H}}(G,K)_R$ can be described as the set of elements for which there exists a finite subset
\[
	X\subset K\backslash G/K
\]
such that
\[
	a_{y,i}=0
\]
for every $y\notin X$ and every $i\in I$.
Thus,
$\overline{\mathscr{H}}(G,K)_R$ may be regarded as a submodule of $\widehat{\mathscr{H}}(G,K)_R$.

Furthermore,
when the multiplication on $\widehat{\mathscr{H}}(G,K)_R$ is restricted to $\overline{\mathscr{H}}(G,K)_R$,
the product of two elements of finite support again has finite support,
and hence belongs to $\overline{\mathscr{H}}(G,K)_R$.
Therefore,
$\overline{\mathscr{H}}(G,K)_R$
may be regarded as a subring of
$\widehat{\mathscr{H}}(G,K)_R$.

The maps on group cohomology induced by the coefficient maps $R\rightarrow R_i$
give rise to maps between derived Hecke algebras
\[
	\mathscr{H}(G,K)_R \rightarrow \mathscr{H}(G,K)_{R_i}.
\]
These maps are compatible with the transition maps of the inverse system. Therefore,
the universal property of the inverse limit gives a map
\[
	\mathscr{H}(G,K)_R \rightarrow \widehat{\mathscr{H}}(G,K)_R.
\]

The maps obtained above form the following diagram:
\[
	\begin{tikzcd}
		\mathscr{H}(G,K)_R \ar[r] \ar[dr]
			& \overline{\mathscr{H}}(G,K)_R \ar[d]\\
		{}
			& \widehat{\mathscr{H}}(G,K)_R.
	\end{tikzcd}
\]
In fact,
this diagram is commutative,
and we obtain the following proposition.

\begin{proposition}\label{proposition7.1}
	The following diagram of graded modules is commutative:
	\[
	\begin{tikzcd}
		\mathscr{H}(G,K)_R \ar[r] \ar[dr]
			&	\overline{\mathscr{H}}(G,K)_R \ar[d]\\
		{}
			& \widehat{\mathscr{H}}(G,K)_R.
	\end{tikzcd}
	\]
	Here the map
	\[
		\mathscr{H}(G,K)_R \rightarrow \overline{\mathscr{H}}(G,K)_R
	\]
	is induced,
	for each $x\in K\backslash G/K$,
	by the compatible family
	\[
		(\phi_i^x)_{i\in I},
	\]
	the map
	\[
		\mathscr{H}(G,K)_R \rightarrow \widehat{\mathscr{H}}(G,K)_R
	\]
	is induced by the universal property of the inverse limit,
	and the map
	\[
		\overline{\mathscr{H}}(G,K)_R \rightarrow \widehat{\mathscr{H}}(G,K)_R
	\]
	is the natural embedding.
	
	Moreover,
	every divisible element of $\mathscr{H}(G,K)_R$ maps to zero in both $\overline{\mathscr{H}}(G,K)_R$ and	$\widehat{\mathscr{H}}(G,K)_R$.
\end{proposition}

\begin{proof}
	By the preceding discussion,
	for every element of $\mathscr{H}(G,K)_R$,
	its image in $\widehat{\mathscr{H}}(G,K)_R$ coincides with the image obtained by passing through $\overline{\mathscr{H}}(G,K)_R$.
	Therefore,
	by the universal property of the inverse limit,
	the required diagram is commutative.
	
	We next consider divisible elements.
	For each $i\in I$,
	the derived Hecke algebra $\mathscr{H}(G,K)_{R_i}$ is an $R_i$-module.
	Since $R_i$ is a finite ring,
	no $R_i$-module has a nonzero divisible element in its underlying additive group.
	Hence	$\mathscr{H}(G,K)_{R_i}$ has no nonzero divisible element.
	
	Since homomorphisms of abelian groups preserve divisibility,
	every divisible element of $\mathscr{H}(G,K)_R$	maps to zero in $\mathscr{H}(G,K)_{R_i}$.
	This holds for every $i\in I$,
	and hence every divisible element of $\mathscr{H}(G,K)_R$ maps to zero in both $\overline{\mathscr{H}}(G,K)_R$ and $\widehat{\mathscr{H}}(G,K)_R$.
\end{proof}

Assume that $I=\mathbb{N}$.
Suppose that for every $x\in K\backslash G/K$ and every degree $n\geqq 0$,
the inverse system
\[
(H^n(K_x,R_i))_{i\in I}
\]
satisfies the Mittag--Leffler condition.
Then,
for group cohomology with coefficients in $R^p$,
where $R$ is equipped with its profinite topology,
we obtain an isomorphism
\[
	\varprojlim_{i\in I}H^*(K_x,R_i) \cong H^*(K_x,R^p).
\]
Consequently,
\[
	\overline{\mathscr{H}}(G,K)_R
		= \bigoplus_{x\in K\backslash G/K}\varprojlim_{i\in I}H^*(K_x,R_i)
		\cong \bigoplus_{x\in K\backslash G/K}H^*(K_x,R^p).
\]

\subsection{The Product on the Derived Hecke Algebra Induced by the Yoneda Product}

The product on the derived Hecke algebra is defined by the Yoneda product, whereas in \cite{koziol2024parahoricheckeextalgebrascharacteristic} it is defined by the graded opposite of the Yoneda product and is computed in terms of group cohomology.
Moreover,
Venkatesh's explicit formula for the product is given only when the coefficient ring is a finite ring satisfying certain conditions.
We therefore show that the graded opposite of the product structure on Venkatesh's function model agrees with the product in \cite{koziol2024parahoricheckeextalgebrascharacteristic},
and that the same product structure can be extended without change to the case of an arbitrary discrete commutative coefficient ring.
Here,
by the graded opposite product we mean the following:
if a product $\alpha*\beta$ is defined for an element $\alpha$ of degree $i$ and an element $\beta$ of degree $j$,
then
\[
	\alpha*_{\mathrm{gop}}\beta
	=
	(-1)^{ij}\beta*\alpha.
\]

For every $x\in G/K$,
let $g_x\in G$ be a representative of $x$,
and write
\[
	K_x
		=	K\cap g_xKg_x^{-1},~K_{x^{-1}}
		= K\cap g_x^{-1}Kg_x.
\]
Moreover,
for any $x,y,z\in G/K$,
put
\[
	G_{x,y}
		=g_xKg_x^{-1}\cap g_yKg_y^{-1},
\]
and
\[
	G_{x,y,z}
		=g_xKg_x^{-1}
		\cap g_yKg_y^{-1}
		\cap g_zKg_z^{-1}.
\]
Let
\[
	\alpha\in H^i(K_x,R),~\beta\in H^j(K_y,R)
\]
be nonzero elements.
Let $h_{\alpha,x}$ and $h_{\beta,y}$ denote the corresponding elements of the function model of the derived Hecke algebra.

In what follows,
via the Shapiro isomorphism
\[
	\operatorname{Sh}^K_{g_x}:
		H^*\left(K,\operatorname{c-Ind}_{K_x}^K(R)\right)
		\cong H^*(K_x,R),
\]
we identify $\alpha$ with the cohomology class appearing in \cite[Proposition 2.10]{koziol2024parahoricheckeextalgebrascharacteristic}.
We similarly identify $\beta\in H^j(K_y,R)$ with the corresponding cohomology class.

We first consider the condition on $z\in G/K$ under which
\[
	h_{\alpha,x}*h_{\beta,y}(z,K)
\]
is nonzero.

\begin{lemma}
	If
	\[
		h_{\alpha,x}*h_{\beta,y}(z,K)\neq 0,
	\]
	then
	\[
		g_zK\subset Kg_yKg_xK.
	\]
\end{lemma}

\begin{proof}
	By definition,
	\begin{align*}
		&h_{\alpha,x}*h_{\beta,y}(z,K)\\
			&=\sum_{g\in K_z\backslash G/K} \operatorname{Cor}^{G_{z,e}}_{G_{z,g,e}}
			\left(
				\operatorname{Res}^{G_{z,g}}_{G_{z,g,e}}
				\bigl(	h_{\alpha,x}(z,gK)\bigr)
				\cup
				\operatorname{Res}^{G_{g,e}}_{G_{z,g,e}}
				\bigl(	h_{\beta,y}(gK,K)\bigr)
			\right).
	\end{align*}
	Here the sum is taken over representatives $g\in G$ chosen from the elements of $K_z\backslash G/K$,
	and we put
	\[
		G_{z,e}=K_z,
	\]
	\[
		G_{z,g}=g_zKg_z^{-1}\cap gKg^{-1},~
		G_{g,e}=gKg^{-1}\cap K,
	\]
	and
	\[
		G_{z,g,e}=K_z\cap gKg^{-1}.
	\]
	
	If
	\[
		h_{\alpha,x}*h_{\beta,y}(z,K)\neq 0,
	\]
		then for some representative $g\in G$ of a double coset,
	\[
		h_{\alpha,x}(z,gK)\neq 0,~h_{\beta,y}(gK,K)\neq 0.
	\]
	Here,
	the condition
	\[
		h_{\beta,y}(gK,K)\neq 0
	\]
	is equivalent to
	\[
		gK\subset Kg_yK.
	\]
	Moreover,
	if
	\[
		h_{\alpha,x}(z,gK)\neq 0,
	\]
	then there exists $g'\in G$ such that
	\[
		g_zK=g'g_xK,~gK=g'K.
	\]
	Therefore,
	if
	\[
		h_{\alpha,x}*h_{\beta,y}(z,K)\neq 0,
	\]
	then
	\[
	g_zK=g'g_xK\subset g'Kg_xK=gKg_xK\subset Kg_yKg_xK.
	\]
\end{proof}

Next,
fix $z\in G/K$ and consider those $gK\in G/K$ satisfying
\[
	h_{\alpha,x}(z,gK)\neq 0,~h_{\beta,y}(gK,K)\neq 0.
\]
Let $\mathcal{S}_z$ denote the support of $h_{\alpha,x}$ and $h_{\beta,y}$ relative to $z$;
that is,
\[
	\mathcal{S}_z=
	\left\{
		gK\in G/K
		\mathrel{}\middle|\mathrel{}
		h_{\alpha,x}(z,gK)\neq 0,\ 
		h_{\beta,y}(gK,K)\neq 0
	\right\}.
\]

We have already seen that
\[
	h_{\beta,y}(gK,K)\neq 0
\]
is equivalent to
\[
	gK\subset Kg_yK.
\]
Moreover,
the condition
\[
	h_{\alpha,x}(z,gK)\neq 0
\]
is equivalent to the existence of some $g'\in G$ satisfying
\[
	g_zK=g'g_xK,~gK=g'K,
\]
and this is in turn equivalent to
\[
	g_zK\subset gKg_xK.
\]
Thus,
$\mathcal{S}_z$ can be written as
\[
	\mathcal{S}_z=
	\left\{
		gK\in G/K
		\mathrel{}\middle|\mathrel{}
		gK\subset Kg_yK,\ 
		g_zK\subset gKg_xK
	\right\}.
\]

\begin{lemma}
	The group $K_z$ acts on $\mathcal{S}_z$ from the left.
\end{lemma}

\begin{proof}
	Take
	\[
	k=g_zk'g_z^{-1}\in K_z
	\]
	and
	\[
	gK\in\mathcal{S}_z.
	\]
	Consider $kgK$.
	Since
	\[
	gK\subset Kg_yK
	\]
	and
	\[
	k\in K_z\subset K,
	\]
	we have
	\[
	kgK\subset Kg_yK.
	\]
	Moreover,
	\[
	kg_zK\subset kgKg_xK.
	\]
	On the other hand,
	\[
	kg_zK
	=
	g_zk'g_z^{-1}g_zK
	=
	g_zk'K
	=
	g_zK.
	\]
	Hence
	\[
	g_zK\subset kgKg_xK.
	\]
	Therefore,
	\[
	kgK\in\mathcal{S}_z.
	\]
\end{proof}

Consequently,
the product on the derived Hecke algebra can be written as
\begin{align*}
	&h_{\alpha,x}*h_{\beta,y}(z,K)\\
		&=
		\sum_{g\in K_z\backslash\mathcal{S}_z}
		\operatorname{Cor}^{G_{z,e}}_{G_{z,g,e}}
		\left(
			\operatorname{Res}^{G_{z,g}}_{G_{z,g,e}}
		\bigl(h_{\alpha,x}(z,gK)\bigr)
		\cup
		\operatorname{Res}^{G_{g,e}}_{G_{z,g,e}}
		\bigl(h_{\beta,y}(gK,K)\bigr)
		\right).
\end{align*}

We next consider the set
\[
K_{y^{-1}}\backslash
\left(g_y^{-1}Kg_z\cap Kg_xK\right)
/K_{z^{-1}}.
\]

\begin{lemma}
	The group $K_{y^{-1}}$ acts from the left,
	and $K_{z^{-1}}$ acts from the right,
	on
	\[
	g_y^{-1}Kg_z\cap Kg_xK.
	\]
\end{lemma}

\begin{proof}
	Take arbitrary elements
	\[
	f\in K_{y^{-1}},~
	h\in K_{z^{-1}},~
	g\in g_y^{-1}Kg_z\cap Kg_xK,
	\]
	and write
	\[
	f=g_y^{-1}k_1g_y,~
	h=g_z^{-1}k_2g_z,~
	g=g_y^{-1}k_3g_z.
	\]
	Since $f,h\in K$,
	we have
	\[
	fg,\ gh,\ fgh\in Kg_xK.
	\]
	Moreover,
	\begin{align*}
		fg
		&=
		g_y^{-1}k_1g_yg_y^{-1}k_3g_z
		=
		g_y^{-1}k_1k_3g_z
		\in
		g_y^{-1}Kg_z,
		\\
		gh
		&=
		g_y^{-1}k_3g_zg_z^{-1}k_2g_z
		=
		g_y^{-1}k_3k_2g_z
		\in
		g_y^{-1}Kg_z.
	\end{align*}
	Thus,
	\[
	fg,\ gh
	\in
	g_y^{-1}Kg_z\cap Kg_xK.\qedhere
	\]
\end{proof}

By the preceding two lemmas,
we may consider the sets
\[
K_z\backslash\mathcal{S}_z
\]
and
\[
K_{y^{-1}}\backslash
\left(g_y^{-1}Kg_z\cap Kg_xK\right)
/K_{z^{-1}}.
\]
In fact,
these two sets are in bijection as follows.

\begin{lemma}
	The map
	\begin{align*}
		\phi:
		K_{y^{-1}}\backslash
		\left(g_y^{-1}Kg_z\cap Kg_xK\right)
		/K_{z^{-1}}
		&\longrightarrow
		K_z\backslash\mathcal{S}_z,
		\\
		K_{y^{-1}}gK_{z^{-1}}
		&\longmapsto
		K_zg_zg^{-1}K
	\end{align*}
	is a bijection.
\end{lemma}

\begin{proof}
	Take
	\[
	g\in g_y^{-1}Kg_z\cap Kg_xK
	\]
	and write
	\[
	g=g_y^{-1}kg_z.
	\]
	Then
	\[
	g_zg^{-1}K
	=
	g_z\left(g_z^{-1}k^{-1}g_y\right)K
	=
	k^{-1}g_yK
	\subset Kg_yK.
	\]
	Moreover,
	since $g\in Kg_xK$,
	\[
	g_zK
	=
	(g_zg^{-1})gK
	\subset
	(g_zg^{-1})Kg_xK.
	\]
	Thus,
	\[
	g_zK\subset(g_zg^{-1})Kg_xK.
	\]
	It follows that
	\[
	g_zg^{-1}K\in\mathcal{S}_z,
	\]
	and hence the displayed rule takes values in
	$K_z\backslash\mathcal{S}_z$.
	
	We next show that $\phi$ is well-defined,
	that is,
	that it is independent of the choice of representative.
	Take
	\[
	f=g_y^{-1}k_1g_y\in K_{y^{-1}}.
	\]
	Considering the image of $fg$ and using
	$f\in K_{y^{-1}}\subset K$,
	we obtain
	\begin{align*}
		K_zg_z(fg)^{-1}K
		&=
		K_zg_zg^{-1}f^{-1}K
		\\
		&=
		K_zg_zg^{-1}K.
	\end{align*}
	Similarly,
	take
	\[
	h=g_z^{-1}k_2g_z\in K_{z^{-1}}.
	\]
	The image of $gh$ is
	\[
	K_zg_z(gh)^{-1}K
	=
	K_zg_zh^{-1}g^{-1}K.
	\]
	Since $h\in K_{z^{-1}}$,
	we have
	\[
	g_zh^{-1}g_z^{-1}\in K_z.
	\]
	Therefore,
	\begin{align*}
		K_zg_zh^{-1}(g_z^{-1}g_z)g^{-1}K
		&=
		K_z(g_zh^{-1}g_z^{-1})g_zg^{-1}K
		\\
		&=
		K_zg_zg^{-1}K.
	\end{align*}
	Hence $\phi$ is well-defined.
	
	We define a map in the opposite direction
	\[
	\psi:
	K_z\backslash\mathcal{S}_z
	\longrightarrow
	K_{y^{-1}}\backslash
	\left(g_y^{-1}Kg_z\cap Kg_xK\right)
	/K_{z^{-1}}
	\]
	as follows.
	Take any
	\[
	gK\in\mathcal{S}_z.
	\]
	Since
	\[
	gK\subset Kg_yK,
	\]
	there exists $k_1\in K$ such that
	\[
	gK=k_1g_yK.
	\]
	We define
	\[
	\psi(K_zgK)
	=
	K_{y^{-1}}(k_1g_y)^{-1}g_zK_{z^{-1}}.
	\]
	
	We have
	\[
	(k_1g_y)^{-1}g_z
	=
	g_y^{-1}k_1^{-1}g_z
	\in
	g_y^{-1}Kg_z.
	\]
	Moreover,
	since
	\[
	g_zK\subset gKg_xK,
	\]
	we have
	\[
	g_zK\subset k_1g_yKg_xK.
	\]
	Thus,
	there exist $k_2,k_3\in K$ such that
	\[
	g_z=k_1g_yk_2g_xk_3.
	\]
	Consequently,
	\begin{align*}
		\psi(K_zgK)
		&=
		K_{y^{-1}}(k_1g_y)^{-1}g_zK_{z^{-1}}
		\\
		&=
		K_{y^{-1}}
		(k_1g_y)^{-1}
		k_1g_yk_2g_xk_3
		K_{z^{-1}}
		\\
		&=
		K_{y^{-1}}k_2g_xk_3K_{z^{-1}}.
	\end{align*}
	Hence the image of $\psi$ belongs to
	\[
	K_{y^{-1}}\backslash
	\left(g_y^{-1}Kg_z\cap Kg_xK\right)
	/K_{z^{-1}}.
	\]
	
	We next show that this map is well-defined,
	that is,
	it is independent of the choice of $k_1$
	and of the representative of the $K_z$-orbit.
	
	We first show that it is independent of the choice of $k_1$.
	Suppose that
	\[
	gK=k_1g_yK=k'g_yK
	\]
	for some $k'\in K$.
	From
	\[
	k_1g_yK=k'g_yK
	\]
	we obtain
	\[
	g_y^{-1}k_1^{-1}k'g_y\in K.
	\]
	Moreover,
	since $k_1^{-1}k'\in K$,
	\[
	g_y^{-1}k_1^{-1}k'g_y
	\in
	g_y^{-1}Kg_y.
	\]
	Thus,
	\[
	g_y^{-1}k_1^{-1}k'g_y
	\in
	K\cap g_y^{-1}Kg_y
	=
	K_{y^{-1}}.
	\]
	Therefore,
	\begin{align*}
		K_{y^{-1}}(k'g_y)^{-1}g_zK_{z^{-1}}
		&=
		K_{y^{-1}}
		(g_y^{-1}k_1^{-1}k'g_y)
		(k'g_y)^{-1}
		g_zK_{z^{-1}}
		\\
		&=
		K_{y^{-1}}
		(g_y^{-1}k_1^{-1})
		g_zK_{z^{-1}}
		\\
		&=
		K_{y^{-1}}
		(k_1g_y)^{-1}
		g_zK_{z^{-1}}.
	\end{align*}
	Hence the definition is independent of the choice of $k_1$.
	
	We next show that it is independent of the representative of the $K_z$-orbit.
	Take any $k\in K_z$.
	If
	\[
	gK=k_1g_yK,
	\]
	then
	\[
	kgK=kk_1g_yK.
	\]
	Hence
	\[
	K_{y^{-1}}
	(kk_1g_y)^{-1}g_z
	K_{z^{-1}}
	=
	K_{y^{-1}}
	(k_1g_y)^{-1}k^{-1}g_z
	K_{z^{-1}}.
	\]
	Since $k\in K_z$,
	\[
	g_z^{-1}k^{-1}g_z\in K_{z^{-1}}.
	\]
	Therefore,
	\begin{align*}
		&
		K_{y^{-1}}
		(k_1g_y)^{-1}k^{-1}g_z
		K_{z^{-1}}
		\\
		&=
		K_{y^{-1}}
		(k_1g_y)^{-1}
		(g_zg_z^{-1})k^{-1}g_z
		K_{z^{-1}}
		\\
		&=
		K_{y^{-1}}
		(k_1g_y)^{-1}
		g_z(g_z^{-1}k^{-1}g_z)
		K_{z^{-1}}
		\\
		&=
		K_{y^{-1}}
		(k_1g_y)^{-1}g_z
		K_{z^{-1}}.
	\end{align*}
	Thus,
	\[
	K_{y^{-1}}
	(kk_1g_y)^{-1}g_z
	K_{z^{-1}}
	=
	K_{y^{-1}}
	(k_1g_y)^{-1}g_z
	K_{z^{-1}},
	\]
	and hence the definition is independent of the representative of the $K_z$-orbit.
	
	Finally,
	we show that $\phi$ and $\psi$ are inverse to each other.
	Take any
	\[
	K_zgK\in K_z\backslash\mathcal{S}_z
	\]
	and choose $k_1\in K$ such that
	\[
	gK=k_1g_yK.
	\]
	Then
	\begin{align*}
		\phi\psi(K_zgK)
		&=
		\phi
		\left(
		K_{y^{-1}}(k_1g_y)^{-1}g_zK_{z^{-1}}
		\right)
		\\
		&=
		K_zg_z
		(g_y^{-1}k_1^{-1}g_z)^{-1}
		K
		\\
		&=
		K_zk_1g_yK
		\\
		&=
		K_zgK.
	\end{align*}
	Thus,
	\[
	\phi\psi=\operatorname{id}.
	\]
	
	Conversely,
	take any
	\[
	K_{y^{-1}}gK_{z^{-1}}
	\in
	K_{y^{-1}}\backslash
	\left(g_y^{-1}Kg_z\cap Kg_xK\right)
	/K_{z^{-1}}.
	\]
	Since
	\[
	g\in g_y^{-1}Kg_z,
	\]
	there exists $k\in K$ such that
	\[
	g=g_y^{-1}kg_z.
	\]
	We then have
	\[
	g_zg^{-1}K=k^{-1}g_yK.
	\]
	Consequently,
	\begin{align*}
		\psi\phi
		(K_{y^{-1}}gK_{z^{-1}})
		&=
		\psi(K_zg_zg^{-1}K)
		\\
		&=
		\psi(K_zk^{-1}g_yK)
		\\
		&=
		K_{y^{-1}}
		(k^{-1}g_y)^{-1}g_z
		K_{z^{-1}}
		\\
		&=
		K_{y^{-1}}
		g_y^{-1}kg_z
		K_{z^{-1}}
		\\
		&=
		K_{y^{-1}}gK_{z^{-1}}.
	\end{align*}
	Thus,
	\[
	\psi\phi=\operatorname{id}.
	\]
	It follows that $\phi$ and $\psi$ are inverse to each other,
	and therefore $\phi$ is a bijection.
\end{proof}

Let $A=K_{y^{-1}}\backslash({g_y}^{-1}Kg_z\cap Kg_xK)/K_{z^{-1}}$.
Using this bijection,
the product on the derived Hecke algebra can again be written as
\begin{align*}
	h_{\alpha,x}*h_{\beta,y}(z,K)
	&=
	\sum_{g\in A}
	\operatorname{Cor}^{G_{z,e}}_{G_{z,g_zg^{-1},e}}
	\Bigl(
	\operatorname{Res}^{G_{z,g_zg^{-1}}}_{G_{z,g_zg^{-1},e}}
	\bigl(
	h_{\alpha,x}(z,g_zg^{-1}K)
	\bigr)
	\\
	&\hspace{40mm}
	\cup
	\operatorname{Res}^{G_{g_zg^{-1},e}}_{G_{z,g_zg^{-1},e}}
	\bigl(
	h_{\beta,y}(g_zg^{-1}K,K)
	\bigr)
	\Bigr).
\end{align*}
Here the sum is taken over representatives
\[
g\in g_y^{-1}Kg_z\cap Kg_xK
\]
of the corresponding double cosets.

For each such $g$,
there exist $k_1,k_2,k_3\in K$ satisfying
\[
g
=
g_y^{-1}k_1^{-1}g_z
=
k_2g_xk_3.
\]
Noting that
\[
g_z=k_1g_yk_2g_xk_3,
\]
we obtain
\begin{align*}
	h_{\beta,y}(g_zg^{-1}K,K)
	&=
	h_{\beta,y}(k_1g_yK,K)
	\\
	&=
	[k_1^{-1}]^*
	h_{\beta,y}(g_yK,K)
	\\
	&=
	[k_1^{-1}]^*\beta,
	\\
	h_{\alpha,x}(z,g_zg^{-1}K)
	&=
	h_{\alpha,x}
	(k_1g_yk_2g_xK,k_1g_yK)
	\\
	&=
	[(k_1g_yk_2)^{-1}]^*
	h_{\alpha,x}(g_xK,K)
	\\
	&=
	[(k_1g_yk_2)^{-1}]^*\alpha.
\end{align*}
Thus,
the product on the derived Hecke algebra can further be written as
\begin{align*}
	&h_{\alpha,x}*h_{\beta,y}(z,K)\\
		&=
		\sum_{g\in A}
		\operatorname{Cor}^{G_{z,e}}_{G_{z,g_zg^{-1},e}}
		\Bigl(
			\operatorname{Res}^{G_{z,g_zg^{-1}}}_{G_{z,g_zg^{-1},e}}
			\left(
				[(k_1g_yk_2)^{-1}]^*\alpha
			\right)
			\cup
			\operatorname{Res}^{G_{g_zg^{-1},e}}_{G_{z,g_zg^{-1},e}}
			\left(
				[k_1^{-1}]^*\beta
			\right)
		\Bigr),
\end{align*}
here, $A=K_{y^{-1}}\backslash (g_y^{-1}Kg_z\cap Kg_xK)/K_{z^{-1}}$.
In particular,
the cup-product factor occurring in each summand is of the form
\[
[(k_1g_yk_2)^{-1}]^*\alpha
\cup
[k_1^{-1}]^*\beta.
\]
On the other hand,
in
\cite[Proposition 2.10]
{koziol2024parahoricheckeextalgebrascharacteristic},
the cup-product factor appearing in the product is of the form
\[
[k_1^{-1}]^*\beta
\cup
[(k_1g_yk_2)^{-1}]^*\alpha.
\]
Since the cup product is graded-commutative,
\[
[(k_1g_yk_2)^{-1}]^*\alpha
\cup
[k_1^{-1}]^*\beta
=
(-1)^{ij}
[k_1^{-1}]^*\beta
\cup
[(k_1g_yk_2)^{-1}]^*\alpha.
\]
Thus,
the product obtained above is precisely the graded opposite of the product in
\cite[Proposition 2.10]
{koziol2024parahoricheckeextalgebrascharacteristic}.

In
\cite{koziol2024parahoricheckeextalgebrascharacteristic},
the product on the derived Hecke algebra is defined as the graded opposite of the Yoneda product.
It follows that,
for an arbitrary coefficient ring,
the product on the derived Hecke algebra determined by the Yoneda product agrees with Venkatesh's product.
We therefore obtain the following.

\begin{proposition}
	Let $(R_i)_{i\in I}$ be an inverse system of finite rings,
	and let
	\[
	R=\varprojlim_{i\in I}R_i
	\]
	be a profinite ring.
	Then,
	for every $i\in I$,
	the natural map
	\[
	\mathscr{H}(G,K)_R
	\longrightarrow
	\mathscr{H}(G,K)_{R_i}
	\]
	is a homomorphism of graded rings.
	Consequently,
	the induced map
	\[
	\mathscr{H}(G,K)_R
	\longrightarrow
	\widehat{\mathscr{H}}(G,K)_R
	\]
	is also a homomorphism of graded rings.
	In particular,
	the diagram in Proposition~\ref{proposition7.1}
	is a commutative diagram of homomorphisms of graded rings.
\end{proposition}

\begin{proof}
	The map
	\[
	\mathscr{H}(G,K)_R
	\longrightarrow
	\mathscr{H}(G,K)_{R_i}
	\]
	is induced by the map between the coefficient rings.
	The resulting maps on group cohomology commute with corestriction,
	restriction,
	conjugation maps,
	and cup products.
	It therefore follows from the product formula above that this map is a homomorphism of graded rings.
	
	Moreover,
	these maps are compatible with the transition maps of the inverse system.
	Hence the induced map
	\[
	\mathscr{H}(G,K)_R
	\longrightarrow
	\widehat{\mathscr{H}}(G,K)_R
	\]
	is also a homomorphism of graded rings.
\end{proof}

\subsection{Product Structure}

Let $R$ be a torsion-free profinite ring.
Throughout this subsection,
we impose the topological hypotheses stated in the Introduction.
We also assume that,
for every $n\geqq 1$ and every $x\in K\backslash G/K$,
the natural maps
\begin{align*}
	H^n(K_x,R^p)\otimes_{\mathbb{Z}}\mathbb{Q}
		&\rightarrow H^n(K_x,R^p\otimes_{\mathbb{Z}}\mathbb{Q})\\
	H^n(K_x,R^p)
		&\rightarrow \varprojlim_{i\in I} H^n(K_x,R_i)
\end{align*}
are isomorphisms.
Set
\[
	X^0:=\overline{\mathscr H}^{\,0}(G,K)_R,
		~X^1:=0,
		~X^r:=\bigl(\overline{\mathscr H}^{\,r}(G,K)_R\bigr)_{\operatorname{tor}}
		~(r\geqq2),
\]
and define
\[
	X:=\bigoplus_{n\geqq0}X^n.
\]
We define the product on $X$ to be the product induced on the inverse limit
by the products on the algebras
$\mathscr H(G,K)_{R_i}$,
and denote this product by $*_p$.

For every $n\geqq0$,
let
\[
\pi^n:
\mathscr H^n(G,K)_R
\longrightarrow
X^n
\]
denote the map obtained above,
and define
\[
\pi:=\bigoplus_{n\geqq0}\pi^n.
\]
Set
\[
Y^r:=\ker(\pi^r).
\]
Then
\[
Y^0=Y^1=0.
\]
We further define
\[
Y:=\bigoplus_{n\geqq0}Y^n.
\]

The map
\[
\pi:
\mathscr H(G,K)_R
\longrightarrow
X
\]
is a surjective homomorphism of graded rings,
and hence there exists a short exact sequence
\[
0
\longrightarrow
Y
\longrightarrow
\mathscr H(G,K)_R
\xrightarrow{\pi}
X
\longrightarrow
0.
\]
By the preceding discussion,
each $Y^r$ is divisible.
Therefore,
the short exact sequence above splits as a sequence of graded abelian groups.

Fix a degree-preserving additive section
\[
s:
X
\longrightarrow
\mathscr H(G,K)_R.
\]
For $a,b\in X$,
define the map $d_s$ by
\begin{align*}
	d_s:
	X\times X
	&\longrightarrow
	\mathscr H(G,K)_R,
	\\
	(a,b)
	&\longmapsto
	s(a)* s(b)-s(a*_pb).
\end{align*}
Since $\pi$ is a homomorphism of graded rings,
we have
\[
d_s(a,b)\in Y.
\]

The following describes the product in
$\mathscr H(G,K)_R$.

\begin{theorem}\label{Theorem4}
	Under the hypotheses of this subsection,
	let $i,j\geqq 0$ and let
	\[
	h_1\in\mathscr H^i(G,K)_R,~	
	h_2\in\mathscr H^j(G,K)_R.
	\]
	Then the following statements hold.
	\begin{enumerate}
		\item
		The image of $h_1* h_2$ in $X$ is given by
		\[
		\pi^{i+j}(h_1* h_2)
		=
		\pi^i(h_1)*_p\pi^j(h_2).
		\]
		
		\item
		Suppose that $i,j>0$.
		Then
		\[
		h_1* h_2
		=
		s\bigl(
		\pi^i(h_1)*_p\pi^j(h_2)
		\bigr)
		+
		d_s\bigl(
		\pi^i(h_1),
		\pi^j(h_2)
		\bigr).
		\]
		
		\item
		Suppose that $i,j>0$ and that
		$h_1\in Y^i$ or $h_2\in Y^j$.
		Then
		\[
		h_1*h_2=0.
		\]
	\end{enumerate}
\end{theorem}

We also obtain the following result concerning the action.

\begin{theorem}
	Under the hypotheses of this subsection,
	let $M$ be a complex in the category of smooth $R[G]$-modules,
	and set
	\[
	\mathscr M
	:=
	 \operatorname{Ext}^*(R[G/K],M),~
	\mathscr M_{ \operatorname{tor}}
	:=
	\bigl(
	 \operatorname{Ext}^*(R[G/K],M)
	\bigr)_{ \operatorname{tor}}.
	\]
	Then $\mathscr M_{ \operatorname{tor}}$ is stable under the action of
	$\mathscr H(G,K)_R$,
	and the restricted action map factors through
	$X\otimes_R\mathscr M_{ \operatorname{tor}}$.
	
	Furthermore,
	the action of elements of positive degree factors as
	\[
	\begin{tikzcd}
		\mathscr H^{>0}(G,K)_R
		\otimes_R
		\mathscr M_{ \operatorname{tor}}
		\ar[dr]
		\ar[r,"\pi^{>0}\otimes\operatorname{id}"]
		&
		X^{>0}
		\otimes_R
		\mathscr M_{ \operatorname{tor}}
		\ar[d]
		\\
		{}
		&
		\mathscr M_{ \operatorname{tor}}.
	\end{tikzcd}
	\]
\end{theorem}

\begin{proof}
	Since the divisible component acts trivially on the torsion part,
	the desired factorization follows.
\end{proof}

% !TeX root = main.tex
\section{Preliminaries for the examples}
\subsection{Preliminaries}

Before computing the examples,
we prove a useful property of the continuous group cohomology $H^*(K_x,R\otimes_\mathbb{Z}\mathbb{Q})$.

Let $K$ be a profinite group,
and let $R$ be a torsion-free profinite ring with a continuous $K$-action.
We consider $R\otimes_\mathbb{Z}\mathbb{Q}$ with the ind-profinite topology defined below.
All cochains in this section are assumed to be continuous.
The following lemma does not require torsion-freeness.

\begin{lemma}
	Let $K$ be a profinite group,
	and let $A$ be an ind-profinite $K$-module.
	Suppose that
	\[
		A=\varinjlim_{i\in \mathbb{N}}A_i
	\]
	for a cofinal direct system $(A_i,\phi_{ij},\mathbb{N})$ of profinite $K$-modules.
	Assume that each canonical map $A_i\rightarrow A$ is a topological embedding.
	Then,
	for every $n\geqq 0$,
	there is a natural isomorphism
	\[
		H^n(K,A)\cong \varinjlim_{i\in \mathbb{N}}H^n(K,A_i).
	\]
\end{lemma}

\begin{proof}
	The maps $\phi_{ij}:A_i\rightarrow A_j$ induce maps of continuous cochain complexes
	\[
		C(\phi_{ij}):C^*(K,A_i)\rightarrow C^*(K,A_j).
	\]
	Thus $(C^*(K,A_i))_i$ forms a direct system of complexes,
	with canonical maps
	\[
		C^*(K,A_i)\rightarrow \varinjlim_i C^*(K,A_i).
	\]
	For each $i\in \mathbb{N}$,
	the canonical map $\phi_i:A_i\rightarrow A$ also induces a map
	\[
		C(\phi_i):C^*(K,A_i)\rightarrow C^*(K,A).
	\]
	These maps are compatible with the transition maps.
	By the universal property of the direct limit,
	they induce a map
	\[
		C(\phi):\varinjlim_{i\in \mathbb{N}}C^*(K,A_i)\rightarrow C^*(K,A).
	\]
	We show that $C(\phi)$ is an isomorphism of complexes.
	
	Let $f\in C^n(K,A)$.
	Since $K^n$ is compact and $f$ is continuous,
	the image $f(K^n)$ is compact.
	By \cite[Corollary~1.6(i)]{boggi2016continuous},
	there is an $i$ such that
	\[
		f(K^n)\subset \phi_i(A_i).
	\]
	Since $\phi_i$ is a topological embedding,
	the map
	\[
		\phi_i:A_i\rightarrow \phi_i(A_i)
	\]
	is a homeomorphism.
	Viewing $f$ as a map with values in $\phi_i(A_i)$,
	we obtain a continuous map
	\[
		f_i:=\phi_i^{-1}\circ f:K^n\rightarrow A_i.
	\]
	Hence $f_i \in C^n(K,A_i)$ and $f=\phi_i\circ f_i$.
	This proves that $C(\phi)$ is surjective in each degree.
	
	To prove injectivity,
	let $f=[f_i]$ be an element of the kernel of $C(\phi)$ in degree $n$,
	with representative $f_i \in C^n(K,A_i)$.
	Then $C(\phi_i)(f_i)=0$.
	Since $\phi_i$ is injective,
	we have $f_i=0$ and hence $f=0$.
	Thus $C(\phi)$ is also injective in each degree.
	We have therefore obtained an isomorphism of complexes
	\[
		\varinjlim_{i\in \mathbb{N}}C^*(K,A_i)\cong C^*(K,A).
	\]
	Filtered colimits are exact in the category of abelian groups and therefore commute with cohomology.
	It follows that,
	for every $n\geqq 0$,
	\begin{align*}
		H^n(K,A)
			&\cong H^n(C^*(K,A))\\
			&\cong H^n(\varinjlim_{i\in \mathbb{N}}C^*(K,A_i))\\
			&\cong \varinjlim_{i\in \mathbb{N}}H^n(C^*(K,A_i))\\
			&\cong \varinjlim_{i\in \mathbb{N}}H^n(K,A_i).
	\end{align*}
\end{proof}

We next consider the functor $-\otimes_{\mathbb{Z}}\mathbb{Q}$.
Recall that
\[
	\mathbb{Q}=\varinjlim_{m\geqq 1}\frac{1}{m!}\mathbb{Z}.
\]
Applying the preceding lemma gives the following result.

\begin{lemma}
	Let $R$ be a torsion-free profinite ring with a continuous $K$-action.
	For each $m\geqq 1$,
	put $R_m:=\frac{1}{m!}R$ and give $R_m$ the profinite topology transported from $R$ by the isomorphism of additive groups
	\[
		R\xrightarrow{\sim} R_m~;~r\mapsto \frac{r}{m!}.
	\]
	Equip $R\otimes_{\mathbb{Z}}\mathbb{Q}$ with the direct limit topology defined by the sequence $(R_m)_{m\geqq 1}$.
	Then,
	for every $n\geqq 0$,
	there is a natural isomorphism
	\[
		H^n(K,R\otimes_{\mathbb{Z}}\mathbb{Q})\cong H^n(K,R^p)\otimes_{\mathbb{Z}}\mathbb{Q}.
	\]
\end{lemma}

\begin{proof}
	Under the isomorphisms $R\cong R_m$ above,
	the inclusion $R_i\rightarrow R_j$ for $i\leqq j$ corresponds to multiplication by $\frac{j!}{i!}$ on $R$.
	Since $R$ is torsion-free,
	this map is injective.
	It is a continuous injection between compact Hausdorff spaces and is therefore a topological embedding.
	
	Moreover,
	\[
		R\otimes_{\mathbb{Z}}\mathbb{Q}=\bigcup_{m\geqq 1}R_m.
	\]
	Thus $(R_m)_{m\geqq 1}$ is a cofinal sequence for $R\otimes_{\mathbb{Z}}\mathbb{Q}$ with respect to the direct limit topology defined above.
	By the preceding lemma,
	for every $n\geqq 0$ we have
	\begin{align*}
		H^n(K,R\otimes_\mathbb{Z}\mathbb{Q})
			&	\cong H^n\left(K,\varinjlim_{m\geqq 1}R_m\right)\\
			&	\cong \varinjlim_{m\geqq 1}H^n(K,R_m).
	\end{align*}
	Under the identifications $H^n(K,R_m)\cong H^n(K,R^p)$,
	the map induced by $R_i\rightarrow R_j$ is also multiplication by $\frac{j!}{i!}$.
	Hence the direct system on the right is identified with
	\[
		H^n(K,R^p) \xrightarrow{\times 2}H^n(K,R^p) \xrightarrow{\times 3} H^n(K,R^p)\rightarrow \cdots.
	\]
	Its direct limit is $H^n(K,R^p)\otimes_{\mathbb{Z}}\mathbb{Q}$,
	which proves the assertion.
\end{proof}

\subsection{Multiplication for abelian locally profinite groups}

Let $G$ be an abelian locally profinite group,
and let $K$ be an open compact subgroup of $G$.
In this subsection,
$R$ is a commutative ring with trivial $G$-action.
Then $K\backslash G/K=G/K$,
and $K_x=K$ for every $x\in G/K$.

For $x,y\in G/K$ and nonzero classes $\alpha,\beta\in H^*(K,R^d)$,
let $h_{\alpha,x}$ and $h_{\beta,y}$ denote the corresponding elements of the derived Hecke algebra.
By the product formula,
we have
\[
	h_{\alpha,x}*h_{\beta,y}(z,K)=\sum_{w\in G/K}h_{\alpha,x}(z,w)\cup h_{\beta,y}(w,K).
\]
Here $h_{\beta,y}(w,K)\neq 0$ if and only if $w=y$,
and $h_{\alpha,x}(z,y)\neq 0$ if and only if $z=yx$.
Thus the product vanishes outside the orbit indexed by $xy$,
and its value on that orbit is $\alpha\cup \beta$.
Therefore,
\[
	h_{\alpha,x}* h_{\beta,y}=h_{\alpha\cup \beta, xy}.
\]
For $n\geqq 0$,
$x\in G/K$,
and $\alpha \in H^n(K,R^d)$,
consider the assignment
\[
	(\alpha,x)\mapsto h_{\alpha,x}.
\]
Extending this assignment $R$-linearly gives an isomorphism of graded rings
\[
	R[G/K]\otimes_R H^*(K,R^d)\xrightarrow{\sim} \mathscr{H}(G,K)_R~;~x\otimes \alpha \mapsto h_{\alpha,x}.
\]
Here $R[G/K]$ is concentrated in degree $0$.

% !TeX root = main.tex
\section{The additive group $\mathbb{Q}_p$}

\subsection{Coefficients in $\mathbb{Z}_p$}
Fix a prime $p$.
Let $G=(\mathbb{Q}_p,+)$,
and let $K$ be a compact open subgroup of $G$.
Since $G$ is abelian,
we have $K\backslash G/K=G/K$ and $K_x=K$ for every $x\in G/K$.
Moreover,
$K$ is topologically isomorphic to $\mathbb{Z}_p$.
We fix such an isomorphism,
which gives
\[
	H^*(K,A)\cong H^*(\mathbb{Z}_p,A)
\]
for every discrete module $A$ with trivial action.

Let $R=\mathbb{Z}_p$.
Consider the short exact sequence of discrete modules
\[
	0 \rightarrow \mathbb{Z}_p^d \rightarrow \mathbb{Q}_p^d \rightarrow \mathbb{Q}_p/\mathbb{Z}_p\rightarrow 0.
\]
Since $\mathbb{Q}_p$ is a $\mathbb{Q}$-vector space,
we have $H^i(K,\mathbb{Q}_p^d)=0$ for $i\geqq 1$.
The associated long exact sequence therefore gives
\[
	H^n(K,R^d)\cong H^{n-1}(K,\mathbb{Q}_p/\mathbb{Z}_p)
\]
for $n\geqq 2$.

The $p$-cohomological dimension of $\mathbb{Z}_p$ is $1$,
so its cohomology with discrete $p$-primary torsion coefficients vanishes in degrees at least $2$.
For each degree $i$,
the inverse system
\[
	\left(H^i(K,\mathbb{Z}/p^m\mathbb{Z})\right)_m
\]
consists of finite groups and hence satisfies the Mittag--Leffler condition.
It follows that
\[
	H^i(K,R^p)\cong \varprojlim_{m} H^i(K,\mathbb{Z}/p^m\mathbb{Z}).
\]
In particular,
$H^i(K,R^p)=0$ for $i\geqq 2$.
Also,
\[
	H^1(K,R^p)
	\cong	\operatorname{Hom}_{\mathrm{cts}}(\mathbb{Z}_p,\mathbb{Z}_p^p)
	\cong \mathbb{Z}_p.
\]
Thus,
for profinite coefficients,
we have
\[
	H^0(K,R^p)\cong H^1(K,R^p)\cong\mathbb{Z}_p.
\]

For discrete coefficients,
let $f\in\operatorname{Hom}_{\mathrm{cts}}(\mathbb{Z}_p,\mathbb{Z}_p^d)$.
The image of $f$ is compact and discrete,
hence finite.
Since $\mathbb{Z}_p$ is torsion-free,
this image is zero.
Thus $H^1(K,R^d)=0$.
In degree $2$,
we have
\[
	H^2(K,R^d)
		\cong	H^1(\mathbb{Z}_p,\mathbb{Z}_p^p) \otimes_{\mathbb{Z}}(\mathbb{Q}_p/\mathbb{Z}_p)
		\cong \mathbb{Q}_p/\mathbb{Z}_p.
\]
Consequently,
\[
	H^i(K,R^d)
		\cong	\begin{cases}
			\mathbb{Z}_p & (i=0),\\
			0 & (i=1),\\
			\mathbb{Q}_p/\mathbb{Z}_p & (i=2),\\
			0 & (i\geqq 3).
		\end{cases}
\]
The derived Hecke algebra is therefore given,
as an abelian group,
by
\begin{align*}
	\mathscr{H}(G,K)_R
	&\cong	\left(\bigoplus_{x\in G/K}\mathbb{Z}_p\right)
		\oplus	\left(\bigoplus_{x\in G/K}\mathbb{Q}_p/\mathbb{Z}_p\right)\\
	&=\mathbb{Z}_p[G/K]
		\oplus(\mathbb{Q}_p/\mathbb{Z}_p)[G/K].
\end{align*}
In the notation introduced earlier,
\[
	X=X^0\cong\mathbb{Z}_p[G/K],~Y=Y^2\cong(\mathbb{Q}_p/\mathbb{Z}_p)[G/K].
\]
Since the only nonzero components occur in degrees $0$ and $2$,
the component
\[
	Y=Y^2=\mathscr{H}^2(G,K)_R
\]
is a square-zero ideal.

Equip $(\mathbb{Q}_p/\mathbb{Z}_p)[G/K]$ with the $\mathbb{Z}_p[G/K]$-module structure defined by
\[
	a[x]\circ b[y]=(ab)[x+y],
\]
where $a\in\mathbb{Z}_p$,
$b\in\mathbb{Q}_p/\mathbb{Z}_p$,
and $x,y\in G/K$.
Under the above identification,
multiplication is given by
\[
	(a_1,b_1)\cdot(a_2,b_2)=(a_1a_2,\ a_1\circ b_2+a_2\circ b_1),
\]
where $a_1,a_2\in\mathbb{Z}_p[G/K]$ and $b_1,b_2\in(\mathbb{Q}_p/\mathbb{Z}_p)[G/K]$.
Placing $\mathbb{Z}_p[G/K]$ in degree $0$ and $(\mathbb{Q}_p/\mathbb{Z}_p)[G/K]$ in degree $2$,
we obtain an isomorphism of graded rings
\[
	\mathscr{H}(G,K)_R \cong \mathbb{Z}_p[G/K]
		\oplus(\mathbb{Q}_p/\mathbb{Z}_p)[G/K]
\]
with this multiplication.

We next consider $\overline{\mathscr{H}}(G,K)_{\mathbb{Z}_p}$.
The calculation
\[
	H^i(K,R^p)\cong
	\begin{cases}
		\mathbb{Z}_p & (i=0,1),\\
		0 & (i\geqq 2)
	\end{cases}
\]
gives
\[
	\overline{\mathscr{H}}^i(G,K)_{\mathbb{Z}_p}
		\cong	\begin{cases}
			\bigoplus_{x\in G/K}\mathbb{Z}_p & (i=0,1),\\
			0 & (i\geqq 2).
		\end{cases}
\]
Since all components of degree at least $2$ vanish,
the product of any two degree-one elements is zero.
Together with the usual action of the degree-zero component,
this gives an isomorphism of graded rings
\[
	\overline{\mathscr{H}}(G,K)_{\mathbb{Z}_p}
		\cong \bigl(\mathbb{Z}_p[G/K]\bigr)[\varepsilon]/(\varepsilon^2).
\]
Here $\mathbb{Z}_p[G/K]$ is placed in degree $0$ and $\varepsilon$ in degree $1$.

The comparison map $\Phi$ is given degreewise by
\[
	\Phi^i:
	\begin{cases}
		\mathbb{Z}_p[G/K]
		\xrightarrow{\mathrm{id}}
		\mathbb{Z}_p[G/K] & (i=0),\\
		0\longrightarrow\mathbb{Z}_p[G/K] & (i=1),\\
		(\mathbb{Q}_p/\mathbb{Z}_p)[G/K]\longrightarrow 0 & (i=2),\\
		0\longrightarrow 0 & (i\geqq 3).
	\end{cases}
\]

Finally,
consider $\widehat{\mathscr{H}}(G,K)_R$.
For each $m\geqq 1$,
the module $\mathbb{Z}/p^m\mathbb{Z}$ is discrete and $p$-primary torsion.
Thus $H^i(K,\mathbb{Z}/p^m\mathbb{Z})=0$ for $i\geqq 2$,
while
\[
	H^1(K,\mathbb{Z}/p^m\mathbb{Z})
		\cong \operatorname{Hom}_{\mathrm{cts}} (\mathbb{Z}_p,\mathbb{Z}/p^m\mathbb{Z})
		\cong\mathbb{Z}/p^m\mathbb{Z}.
\]
It follows that
\[
	\mathscr{H}(G,K)_{\mathbb{Z}/p^m\mathbb{Z}}
		\cong \bigl((\mathbb{Z}/p^m\mathbb{Z})[G/K]\bigr) [\varepsilon]/(\varepsilon^2).
\]
Choosing the degree-one generators compatibly with the transition maps and taking inverse limits,
we obtain
\begin{align*}
	\widehat{\mathscr{H}}(G,K)_R
		&\cong \varprojlim_m \mathscr{H}(G,K)_{\mathbb{Z}/p^m\mathbb{Z}}\\
		&\cong \left( \varprojlim_m \bigl( (\mathbb{Z}/p^m\mathbb{Z})[G/K]\bigr) \right) [\varepsilon]/(\varepsilon^2).
\end{align*}
Again,
the coefficient ring on the right is placed in degree $0$ and $\varepsilon$ in degree $1$.
In particular,
\[
	\widehat{\mathscr{H}}^0(G,K)_R
		\cong
		\varprojlim_m
		\bigl((\mathbb{Z}/p^m\mathbb{Z})[G/K]\bigr).
\]
Since $G/K$ is an infinite discrete group, we have a strict inclusion
\[
	\mathbb{Z}_p[G/K]
		\subsetneq
		\varprojlim_m
		\bigl((\mathbb{Z}/p^m\mathbb{Z})[G/K]\bigr).
\]
Thus the natural map
\[
	\Psi:
		\overline{\mathscr{H}}(G,K)_R
		\hookrightarrow
		\widehat{\mathscr{H}}(G,K)_R
\]
is injective but not surjective.

\subsection{Coefficients in $\mathbb{Z}_\ell$}

Let $G=(\mathbb{Q}_p,+)$,
and let $K$ be a compact open subgroup of $G$.
Let $\ell \ne p$ be a prime, and put $R=\mathbb{Z}_\ell$.

We first compute $H^i(K,R^d)$.
The group $K$ is topologically isomorphic to $\mathbb{Z}_p$ and satisfies $\operatorname{cd}_\ell(K)=0$.
For every open subgroup $U\subseteq K$,
the quotient $K/U$ is a finite $p$-group.
Since $p$ is invertible in $\mathbb{Z}_\ell$,
the cohomology $H^*(K/U,\mathbb{Z}_\ell)$ is concentrated in degree $0$.
Passing to the direct limit over $U$,
we obtain
\[
	H^i(K,R^d)\cong
	\begin{cases}
		\mathbb{Z}_\ell & (i=0),\\
		0 & (i\geqq 1).
	\end{cases}
\]
Hence the derived Hecke algebra is also concentrated in degree $0$,
with
\[
	\mathscr{H}^i(G,K)_{\mathbb{Z}_\ell}
	\cong
	\begin{cases}
		\mathbb{Z}_\ell[G/K] & (i=0),\\
		0 & (i\geqq 1).
	\end{cases}
\]
Similarly,
\[
	\overline{\mathscr{H}}^i(G,K)_{\mathbb{Z}_\ell}
	\cong
	\begin{cases}
		\mathbb{Z}_\ell[G/K] & (i=0),\\
		0 & (i\geqq 1).
	\end{cases}
\]
Since the degree-zero component of $\Phi$ is the identity,
we obtain an isomorphism
\[
\Phi:
\mathscr{H}(G,K)_{\mathbb{Z}_\ell}
\xrightarrow{\sim}
\overline{\mathscr{H}}(G,K)_{\mathbb{Z}_\ell}.
\]

For the completed algebra, note that
$H^*(K,\mathbb{Z}/\ell^m\mathbb{Z})$
is concentrated in degree $0$, where it is isomorphic to
$\mathbb{Z}/\ell^m\mathbb{Z}$.
Thus
\begin{align*}
	\widehat{\mathscr{H}}(G,K)_{\mathbb{Z}_\ell}
	&\cong
	\varprojlim_m
	\bigl(H^*(K,\mathbb{Z}/\ell^m\mathbb{Z})[G/K]\bigr)\\
	&\cong
	\varprojlim_m
	\bigl((\mathbb{Z}/\ell^m\mathbb{Z})[G/K]\bigr).
\end{align*}
As in the preceding subsection, the infinitude of $G/K$ gives
\[
	\mathbb{Z}_\ell[G/K]
	\subsetneq
	\varprojlim_m
	\bigl((\mathbb{Z}/\ell^m\mathbb{Z})[G/K]\bigr).
\]
Therefore the natural map
\[
	\Psi:
	\overline{\mathscr{H}}(G,K)_{\mathbb{Z}_\ell}
	\hookrightarrow
	\widehat{\mathscr{H}}(G,K)_{\mathbb{Z}_\ell}
\]
is injective but not surjective.

\subsection{Coefficients in $\widehat{\mathbb{Z}}$}

Put
\[
R=\widehat{\mathbb{Z}}
=\varprojlim_n\mathbb{Z}/n\mathbb{Z},
\]
where the positive integers are ordered by divisibility.
We have
\[
R\otimes_{\mathbb{Z}}\mathbb{Q}\cong\mathbb{A}_f,
\]
where $\mathbb{A}_f$ is the ring of finite adeles of $\mathbb{Q}$.
There is a short exact sequence of discrete modules with trivial $K$-action
\[
0
\longrightarrow\widehat{\mathbb{Z}}^d
\longrightarrow\mathbb{A}_f^d
\longrightarrow\mathbb{Q}/\mathbb{Z}
\longrightarrow 0.
\]
We also have
\[
\widehat{\mathbb{Z}}
\otimes_{\mathbb{Z}}(\mathbb{Q}/\mathbb{Z})
\cong\mathbb{Q}/\mathbb{Z}
\cong\bigoplus_\ell\mathbb{Q}_\ell/\mathbb{Z}_\ell.
\]

Since $\mathbb{Z}_p$ is pro-$p$,
every continuous homomorphism from $\mathbb{Z}_p$ to this discrete module has finite $p$-group image.
Hence
\begin{align*}
	H^1\left(
	\mathbb{Z}_p,
	\widehat{\mathbb{Z}}
	\otimes_{\mathbb{Z}}(\mathbb{Q}/\mathbb{Z})
	\right)
	&\cong
	\operatorname{Hom}_{\mathrm{cts}}\left(
	\mathbb{Z}_p,
	\widehat{\mathbb{Z}}
	\otimes_{\mathbb{Z}}(\mathbb{Q}/\mathbb{Z})
	\right)\\
	&\cong
	\operatorname{Hom}_{\mathrm{cts}}
	(\mathbb{Z}_p,\mathbb{Q}_p/\mathbb{Z}_p)\\
	&\cong\mathbb{Q}_p/\mathbb{Z}_p,
\end{align*}
where the last isomorphism is given by evaluation at $1$.

Since $\mathbb{A}_f^d$ is a discrete $\mathbb{Q}$-vector space,
\[
	H^i(K,\mathbb{A}_f^d)=0~(i\geqq 1).
\]
The long exact sequence therefore gives
\[
	H^2(K,\widehat{\mathbb{Z}}^d) \cong\mathbb{Q}_p/\mathbb{Z}_p,
\]
and we obtain
\[
	H^i(K,\widehat{\mathbb{Z}}^d)\cong
	\begin{cases}
		\widehat{\mathbb{Z}} & (i=0),\\
		0 & (i=1),\\
		\mathbb{Q}_p/\mathbb{Z}_p & (i=2),\\
		0 & (i\geqq 3).
	\end{cases}
\]

For profinite coefficients,
the inverse system
\[
	\bigl(H^i(K,\mathbb{Z}/n\mathbb{Z})\bigr)_n
\]
consists of finite groups in each degree $i$ and hence satisfies the Mittag--Leffler condition.
Thus
\[
	H^i(K,\widehat{\mathbb{Z}}^p)
		\cong \varprojlim_n H^i(K,\mathbb{Z}/n\mathbb{Z}).
\]
Only the $p$-primary components contribute in degree $1$,
giving
\[
	H^1(K,\widehat{\mathbb{Z}}^p)\cong\mathbb{Z}_p.
\]
Consequently,
\[
	H^i(K,\widehat{\mathbb{Z}}^p)\cong
	\begin{cases}
		\widehat{\mathbb{Z}} & (i=0),\\
		\mathbb{Z}_p & (i=1),\\
		0 & (i\geqq 2).
	\end{cases}
\]
Since $G/K$ is infinite,
we again have a strict inclusion
\[
	\widehat{\mathbb{Z}}[G/K]\subsetneq \varprojlim_n \bigl( (\mathbb{Z}/n\mathbb{Z})[G/K]\bigr).
\]

Combining these calculations,
we obtain
\begin{align*}
	\mathscr{H}^i(G,K)_{\widehat{\mathbb{Z}}}
		&\cong
		\begin{cases}
			\widehat{\mathbb{Z}}[G/K] & (i=0),\\
			0 & (i=1),\\
			(\mathbb{Q}_p/\mathbb{Z}_p)[G/K] & (i=2),\\
			0 & (i\geqq 3),
		\end{cases}\\
	\overline{\mathscr{H}}^i(G,K)_{\widehat{\mathbb{Z}}}
		&\cong
		\begin{cases}
			\widehat{\mathbb{Z}}[G/K] & (i=0),\\
			\mathbb{Z}_p[G/K] & (i=1),\\
			0 & (i\geqq 2),
		\end{cases}\\
		\widehat{\mathscr{H}}^i(G,K)_{\widehat{\mathbb{Z}}}
		&\cong
		\begin{cases}
			\displaystyle
			\varprojlim_n
			\bigl((\mathbb{Z}/n\mathbb{Z})[G/K]\bigr)
			& (i=0),\\
			\displaystyle
			\varprojlim_m
			\bigl((\mathbb{Z}/p^m\mathbb{Z})[G/K]\bigr)
			& (i=1),\\
			0 & (i\geqq 2).
		\end{cases}
\end{align*}
In each of these algebras,
the product of any two elements of positive degree is zero.
The action of the degree-zero component on the positive-degree components is the usual group-module action through the projection to the $p$-component of the coefficient ring.
For $\widehat{\mathscr{H}}(G,K)_{\widehat{\mathbb{Z}}}$,
this action is obtained by taking the inverse limit of the corresponding actions with finite coefficients.

Under these identifications, the natural map
\[
\mathscr{H}(G,K)_R
\longrightarrow
\overline{\mathscr{H}}(G,K)_R
\]
is the identity in degree $0$ and zero in all positive degrees.
% !TeX root = main.tex
\section{An example: $\operatorname{GL}_2(\mathbb{Q}_p)$ with $\mathbb{Z}_p$-coefficients}

Unless otherwise stated,
we assume that $p>3$ throughout this section.
Let
\[
	G=\operatorname{GL}_2(\mathbb{Q}_p),~K=\operatorname{GL}_2(\mathbb{Z}_p).
\]
Then $K$ is a maximal compact open subgroup of $G$,
and the Cartan decomposition gives
\[
	G=\coprod_{\substack{a,b\in\mathbb{Z}\\a\geqq b}}K
	\begin{pmatrix}
		p^a & 0\\
		0 & p^b
	\end{pmatrix}K.
\]
Let $A$ denote the set of representatives appearing in this decomposition.
We have
\[
	\mathscr{H}(G,K)_R \cong \bigoplus_{x\in A}H^*(K_x,R^d).
\]
For
\[
	g_{a,b}=
	\begin{pmatrix}
		p^a & 0\\
		0 & p^b
	\end{pmatrix}\in A,
\]
we may write
\[
	g_{a,b}=p^b
	\begin{pmatrix}
		p^{a-b} & 0\\
		0 & 1
	\end{pmatrix},~
	g_{a,b}^{-1}=p^{-b}
	\begin{pmatrix}
		p^{b-a} & 0\\
		0 & 1
	\end{pmatrix}.
\]
Thus
\[
	g_{a,b}
	\begin{pmatrix}
		x & y\\
		z & w
	\end{pmatrix}g_{a,b}^{-1}
	=
	\begin{pmatrix}
		x & p^{a-b}y\\
		p^{b-a}z & w
	\end{pmatrix}.
\]
It follows that $K_{g_{a,b}}$ consists of the elements of $K$ whose $(1,2)$-entry belongs to $p^{a-b}\mathbb{Z}_p$.
In particular,
this subgroup depends only on $n=a-b$.
We therefore write
\[
	K_{g_{a,b}}=K_n:=
	\left\{
		\begin{pmatrix}
			x & y\\
			z & w
		\end{pmatrix}\in\operatorname{GL}_2(\mathbb{Z}_p)
		\;\middle|\;
		y\in p^n\mathbb{Z}_p
	\right\}.
\]

\subsection{The structure of $K_n$ for $p>2$}

In this subsection,
we assume only that $p>2$.
We first determine the abelianization $K_n^{\mathrm{ab}}$.
For $n=0$,
it is well known that
\[
	K_0^{\mathrm{ab}}=\operatorname{GL}_2(\mathbb{Z}_p)^{\mathrm{ab}}
		\cong\mathbb{Z}_p^\times.
\]
The following lemma extends this description to all $n\geqq 0$.

\begin{lemma}
	Suppose that $p>2$ and $n\geqq 0$.
	Then
	\[
		K_n^{\mathrm{ab}}
			\cong\mathbb{Z}_p^\times \times(\mathbb{Z}/p^n\mathbb{Z})^\times.
	\]
	For $n=0$, the second factor is understood to be the trivial group.
\end{lemma}

\begin{proof}
	The case $n=0$ follows from $\operatorname{GL}_2(\mathbb{Z}_p)^{\mathrm{ab}}\cong\mathbb{Z}_p^\times$.
	Suppose that $n\geqq 1$.
	For
	\[
		g=
		\begin{pmatrix}
			x & y\\
			z & w
		\end{pmatrix}\in K_n,
	\]
	define
	\begin{align*}
		f:K_n&\rightarrow
			\mathbb{Z}_p^\times\times
			\mathbb{Z}_p^\times/(1+p^n\mathbb{Z}_p)\\
		g&\mapsto (\det g,\overline{x}),
	\end{align*}
	where $\overline{x}=x(1+p^n\mathbb{Z}_p)$.
	Since $\det g=xw-yz\in\mathbb{Z}_p^\times$ and $y\in p^n\mathbb{Z}_p$, we have $xw\in\mathbb{Z}_p^\times$.
	Hence $x,w\in\mathbb{Z}_p^\times$.
	If
	\[
		g'=
		\begin{pmatrix}
			x' & y'\\
			z' & w'
		\end{pmatrix}\in K_n,
	\]
	then the $(1,1)$-entry of $gg'$ is $xx'+yz'$,
	and
	\[
		xx'+yz'\equiv xx'\pmod{p^n\mathbb{Z}_p}.
	\]
	Thus $f$ is a continuous homomorphism.
	Given $u\in\mathbb{Z}_p^\times$ and $\overline{x}\in\mathbb{Z}_p^\times/(1+p^n\mathbb{Z}_p)$,
	choose a representative $x\in\mathbb{Z}_p^\times$.
	The matrix
	\[
		\begin{pmatrix}
			x & 0\\
			0 & u/x
		\end{pmatrix}
	\]
	belongs to $K_n$ and maps to $(u,\overline{x})$.
	Therefore $f$ is surjective.
	It remains to show that $\ker f=[K_n,K_n]$.
	We have
	\[
		\ker f=
			\left\{
				g=
				\begin{pmatrix}
					x & y\\
					z & w
				\end{pmatrix}\in K_n
			\;\middle|\;
			\det g=1,\quad x\in1+p^n\mathbb{Z}_p
		\right\}.
	\]
	Since the target of $f$ is abelian,
	$[K_n,K_n]\subseteq\ker f$.
	We prove the reverse inclusion.
	
	Let
	\[
		g=
		\begin{pmatrix}
			x & y\\
			z & w
		\end{pmatrix}\in\ker f.
	\]
	Multiplication by upper and lower triangular matrices in $K_n$ gives
	\begin{align*}
		&\begin{pmatrix}
			1 & 0\\
			-z/x & 1
		\end{pmatrix}
			\begin{pmatrix}
				x & y\\
				z & w
			\end{pmatrix}
			\begin{pmatrix}
				1 & -y/x\\
				0 & 1
			\end{pmatrix}\\
		=
		&\begin{pmatrix}
			x & y\\
			0 & (xw-zy)/x
		\end{pmatrix}
		\begin{pmatrix}
			1 & -y/x\\
			0 & 1
		\end{pmatrix}\\
		=
			&\begin{pmatrix}
				x & 0\\
				0 & \det(g)/x
			\end{pmatrix}.
	\end{align*}
	Hence
	\[
		g=
		\begin{pmatrix}
			1 & 0\\
			z/x & 1
		\end{pmatrix}
		\begin{pmatrix}
			x & 0\\
			0 & \det(g)/x
		\end{pmatrix}
		\begin{pmatrix}
			1 & y/x\\
			0 & 1
		\end{pmatrix}.
	\]
	For $s\in p^n\mathbb{Z}_p$,
	we have
	\begin{align*}
		\left[
			\begin{pmatrix}
				2 & 0\\
				0 & 1
			\end{pmatrix},
			\begin{pmatrix}
				1 & s\\
				0 & 1
			\end{pmatrix}
		\right]
		&=
			\begin{pmatrix}
				2 & 0\\
				0 & 1
			\end{pmatrix}
			\begin{pmatrix}
				1 & s\\
				0 & 1
			\end{pmatrix}
			\begin{pmatrix}
				1/2 & 0\\
				0 & 1
			\end{pmatrix}
			\begin{pmatrix}
				1 & -s\\
				0 & 1
			\end{pmatrix}\\
		&=
			\begin{pmatrix}
				2 & 2s\\
				0 & 1
			\end{pmatrix}
			\begin{pmatrix}
				1/2 & -s/2\\
				0 & 1
			\end{pmatrix}\\
		&=
			\begin{pmatrix}
				1 & s\\
				0 & 1
			\end{pmatrix}.
	\end{align*}
	Similarly,
	for $t\in\mathbb{Z}_p$,
	\begin{align*}
		\left[
			\begin{pmatrix}
				2 & 0\\
				0 & 1
			\end{pmatrix},
			\begin{pmatrix}
				1 & 0\\
				-2t & 1
			\end{pmatrix}
		\right]
	&=
		\begin{pmatrix}
			2 & 0\\
			0 & 1
		\end{pmatrix}
		\begin{pmatrix}
			1 & 0\\
			-2t & 1
		\end{pmatrix}
		\begin{pmatrix}
			1/2 & 0\\
			0 & 1
		\end{pmatrix}
		\begin{pmatrix}
			1 & 0\\
			2t & 1
		\end{pmatrix}\\
	&=
		\begin{pmatrix}
			2 & 0\\
			-2t & 1
		\end{pmatrix}
		\begin{pmatrix}
			1/2 & 0\\
			2t & 1
		\end{pmatrix}\\
	&=
		\begin{pmatrix}
			1 & 0\\
			t & 1
		\end{pmatrix}.
	\end{align*}
	Thus every upper or lower unipotent matrix in $K_n$ belongs to $[K_n,K_n]$.
	Since $\det g=1$,
	the diagonal factor is
	\[
		\begin{pmatrix}
			x & 0\\
			0 & \det(g)/x
		\end{pmatrix}=
		\begin{pmatrix}
			x & 0\\
			0 & x^{-1}
		\end{pmatrix}.
	\]
	This matrix can be written as
	\[
	\begin{pmatrix}
		x & 0\\
		0 & x^{-1}
	\end{pmatrix}
	=
	\begin{pmatrix}
		1 & x-1\\
		0 & 1
	\end{pmatrix}
	\begin{pmatrix}
		1 & 0\\
		1 & 1
	\end{pmatrix}
	\begin{pmatrix}
		1 & x^{-1}-1\\
		0 & 1
	\end{pmatrix}
	\begin{pmatrix}
		1 & 0\\
		-x & 1
	\end{pmatrix}.
	\]
	As $x\in1+p^n\mathbb{Z}_p$,
	both $x-1$ and $x^{-1}-1$ belong to $p^n\mathbb{Z}_p$.
	All four factors therefore lie in $[K_n,K_n]$.
	It follows that $g\in[K_n,K_n]$,
	as required.
\end{proof}

For $n\geqq 1$,
we have
\[
	(\mathbb{Z}/p^n\mathbb{Z})^\times
		\cong\mathbb{Z}/(p-1)\mathbb{Z} \times\mathbb{Z}/p^{n-1}\mathbb{Z}.
\]
Thus,
for any coefficient module $R$ with trivial $K_n$-action,
the first cohomology is given by
\begin{align*}
	H^1(K_n,R)
	&=\operatorname{Hom}_{\mathrm{cts}}(K_n,R)\\
	&\cong\operatorname{Hom}_{\mathrm{cts}}
	\bigl(
		\mathbb{Z}_p^\times\times\mathbb{Z}/(p-1)\mathbb{Z}
		\times\mathbb{Z}/p^{n-1}\mathbb{Z},R
	\bigr)\\
	&\cong\operatorname{Hom}_{\mathrm{cts}}(\mathbb{Z}_p^\times,R)
		\times\operatorname{Hom}_{\mathrm{cts}}(\mathbb{Z}/(p-1)\mathbb{Z},R)
		\times\operatorname{Hom}_{\mathrm{cts}}(\mathbb{Z}/p^{n-1}\mathbb{Z},R).
\end{align*}
If $R$ is torsion-free,
every homomorphism from a finite group to $R$ is zero.
Consequently,
for every $n\geqq 0$,
\[
	\operatorname{Hom}_{\mathrm{cts}}(K_n,R)
		\cong\operatorname{Hom}_{\mathrm{cts}}(\mathbb{Z}_p^\times,R).
\]

To compute the cohomology in higher degrees,
we use the following decomposition of $K_n$.

\begin{lemma}
	For $n\geqq 0$,
	put
	\[
		L_n=\{g\in K_n\mid\det g\in\mu_{p-1}\}.
	\]
	Then there is a natural direct product decomposition
	\[
		K_n\cong L_n\times(1+p\mathbb{Z}_p).
	\]
\end{lemma}

\begin{proof}
	For $g\in K_n$,
	write $\det g=uv$ uniquely with $u\in1+p\mathbb{Z}_p$ and $v\in\mu_{p-1}$.
	Under the isomorphism $1+p\mathbb{Z}_p\cong\mathbb{Z}_p$,
	the squaring map corresponds to multiplication by $2$.
	Since $p>2$,
	this is an automorphism.
	Thus every element of $1+p\mathbb{Z}_p$ has a unique square root in $1+p\mathbb{Z}_p$.
	We obtain a continuous homomorphism
	\[
		\sqrt{\det}:K_n\longrightarrow1+p\mathbb{Z}_p,~ g\longmapsto\sqrt{u},
	\]
	and an exact sequence
	\[
	1\longrightarrow L_n\longrightarrow K_n
	\xrightarrow{\sqrt{\det}}1+p\mathbb{Z}_p
	\longrightarrow1.
	\]
	The map
	\[
	1+p\mathbb{Z}_p\longrightarrow K_n,~
	u\longmapsto\begin{pmatrix}
		u & 0\\
		0 & u
	\end{pmatrix}
	\]
	is a continuous section.
	More explicitly,
	for
	\[
	g=\begin{pmatrix}
		x & y\\
		z & w
	\end{pmatrix},~ \det g=uv,
	\]
	we have
	\[
	g=
	\begin{pmatrix}
		\sqrt{u} & 0\\
		0 & \sqrt{u}
	\end{pmatrix}
	\begin{pmatrix}
		x/\sqrt{u} & y/\sqrt{u}\\
		z/\sqrt{u} & w/\sqrt{u}
	\end{pmatrix},
	\]
	where the second factor belongs to $L_n$.
	Since the image of the section is central in $K_n$,
	the decomposition is a direct product.
\end{proof}

In particular,
\[
K_n\cong L_n\times(1+p\mathbb{Z}_p)
\cong L_n\times\mathbb{Z}_p.
\]

\subsection{The structure of $L_n$ as a $p$-adic Lie group}

We continue to assume that $p>2$.
For the cohomology computations below, we introduce the following notation.

\begin{definition}
	For $n\geqq 0$,
	put
	\[
	F_n:=\begin{cases}
		0 & (n=0,1),\\
		\mathbb{Z}/p^{n-1}\mathbb{Z} & (n\geqq 2).
	\end{cases}
	\]
	For $m\geqq 1$,
	put
	\[
	F_{n,m}:=\begin{cases}
		0 & (n=0,1),\\
		\mathbb{Z}/p^{\min\{m,n-1\}}\mathbb{Z} & (n\geqq 2).
	\end{cases}
	\]
\end{definition}

With this notation,
\[
F_n^\vee\cong F_n,~
F_{n,m}^\vee\cong F_{n,m},~
\operatorname{Hom}(F_n,\mathbb{Z}/p^m\mathbb{Z})\cong F_{n,m}.
\]

The group $L_n$ is closed in $K_n$,
since it is the kernel of a continuous homomorphism.
Its dimension as a $p$-adic Lie group need not be the same as that of $K_n$.
More precisely,
we have the following.

\begin{lemma}
	For every $n\geqq 0$, the group $L_n$ is a $p$-adic Lie group of dimension $3$.
\end{lemma}

\begin{proof}
	It suffices to show that $L_n$ contains an open subgroup of $\operatorname{SL}_2(\mathbb{Q}_p)$ as an open subgroup.
	The subgroup $K_n\cap\operatorname{SL}_2(\mathbb{Z}_p)$ is the kernel of
	\[
	\det:L_n\longrightarrow\mu_{p-1},
	\]
	and is therefore open in $L_n$.
	Since $K_n$ is open in $\operatorname{GL}_2(\mathbb{Q}_p)$,
	this subgroup is also open in $\operatorname{SL}_2(\mathbb{Q}_p)$.
	Hence $L_n$ and $\operatorname{SL}_2(\mathbb{Q}_p)$ have the same dimension,
	namely $3$.
\end{proof}

We next determine $L_n^{\mathrm{ab}}$,
which will be used to compute the first cohomology of $L_n$.
By the direct product decomposition above, any $g,g'\in K_n$ can be written as
\[
g=g_lg_u,~
g'=g_l'g_u',
\]
with $g_l,g_l'\in L_n$ and $g_u,g_u'\in1+p\mathbb{Z}_p$,
where the latter group is embedded as scalar matrices.
The factors $g_u$ and $g_u'$ are central,
so
\[
[g,g']=[g_lg_u,g_l'g_u']=[g_l,g_l']\in[L_n,L_n].
\]
It follows that
\[
[L_n,L_n]=[K_n,K_n].
\]
Thus $L_n^{\mathrm{ab}}$ is a subgroup of
\[
K_n/[K_n,K_n]
\cong\mathbb{Z}_p^\times\times(\mathbb{Z}/p^n\mathbb{Z})^\times.
\]
Since $L_n$ consists of the elements whose determinant belongs to $\mu_{p-1}$,
we obtain
\[
L_0^{\mathrm{ab}}\cong\mu_{p-1},~
L_n^{\mathrm{ab}}
\cong\mu_{p-1}\times(\mathbb{Z}/p^n\mathbb{Z})^\times
\quad(n\geqq 1).
\]
We summarize this as follows.

\begin{lemma}
	For every $n\geqq 0$,
	\[
	L_n^{\mathrm{ab}}\cong\begin{cases}
		\mathbb{Z}/(p-1)\mathbb{Z} & (n=0),\\
		\mathbb{Z}/(p-1)\mathbb{Z}
		\times(\mathbb{Z}/p^n\mathbb{Z})^\times & (n\geqq 1).
	\end{cases}
	\]
	For $n\geqq 1$,
	the second group is isomorphic to
	\[
	\bigl(\mathbb{Z}/(p-1)\mathbb{Z}\bigr)^2
	\times\mathbb{Z}/p^{n-1}\mathbb{Z}.
	\]
	In particular,
	the $p$-primary component of $L_n^{\mathrm{ab}}$ is isomorphic to $F_n$.
\end{lemma}

Let $A$ be a discrete module with trivial action.
The decomposition $K_n\cong L_n\times\mathbb{Z}_p$ gives a spectral sequence
\[
H^i\bigl(\mathbb{Z}_p,H^j(L_n,A)\bigr)
\Longrightarrow H^{i+j}(K_n,A).
\]
In low degrees,
this gives an exact sequence
\[
0\longrightarrow H^1(\mathbb{Z}_p,A)
\longrightarrow H^1(K_n,A)
\longrightarrow H^1(L_n,A)^{\mathbb{Z}_p}
\longrightarrow H^2(\mathbb{Z}_p,A).
\]
Since $\mathbb{Z}_p$ has cohomological dimension $1$,
its cohomology with discrete torsion coefficients vanishes in degrees greater than $1$.

We now consider the dualizing modules of $L_n$.
A compact $p$-adic Lie group with no $p$-torsion is a Poincar\'e group.
For $\operatorname{GL}_2(\mathbb{Z}_p)$,
we use the following fact.

\begin{lemma}[{\cite[p.~276]{venjakob2002structure}}]
	If $p>3$, then $\operatorname{GL}_2(\mathbb{Z}_p)$ has no $p$-torsion.
\end{lemma}

Since $L_n\subseteq\operatorname{GL}_2(\mathbb{Z}_p)$,
the group $L_n$ also has no $p$-torsion when $p>3$.
From now on,
we assume that $p>3$.

Let $\mathfrak{g}$ be an $L_n$-stable $\mathbb{Z}_p$-lattice in $\operatorname{Lie}(L_n)$,
equipped with the adjoint action.
Write $I_n$ for the discrete dualizing module of $L_n$ and $D_n$ for its compact dualizing module,
and put
\[
\mathfrak{g}^*
=\operatorname{Hom}_{\mathbb{Z}_p}(\mathfrak{g},\mathbb{Z}_p).
\]
By \cite[Proposition~4.40]{beaudry2022dualizing},
there is an isomorphism
\[
D_n\cong\bigwedge_{\mathbb{Z}_p}^3\mathfrak{g}^*
\]
of right $L_n$-modules.
The action is described in terms of the determinant of the adjoint action in \cite[Remark~4.23]{beaudry2022dualizing}.
The argument in the proof of \cite[Corollary~5.2]{kohlhaase2017smooth} shows that this action is trivial.
Moreover,
\[
I_n\cong\operatorname{Hom}_{\mathrm{cts}}
(D_n,\mathbb{Q}_p/\mathbb{Z}_p).
\]
Since $D_n$ has trivial $L_n$-action,
so does $I_n$.
Hence
\[
D_n\cong\mathbb{Z}_p,~
I_n\cong\mathbb{Q}_p/\mathbb{Z}_p,
\]
with trivial action in both cases.
For $\mathbb{Z}_p^p$ with trivial $L_n$-action and its profinite topology,
we also have
\[
\operatorname{Hom}_{\mathrm{cts}}(\mathbb{Z}_p^p,I_n)
\cong I_n\cong\mathbb{Q}_p/\mathbb{Z}_p.
\]
Thus \cite[Theorem~4.26]{beaudry2022dualizing} gives a perfect pairing
\[
H^i(L_n,\mathbb{Z}_p^p)
\otimes H^{3-i}(L_n,\mathbb{Q}_p/\mathbb{Z}_p)
\longrightarrow\mathbb{Q}_p/\mathbb{Z}_p.
\]
In particular,
\[
\operatorname{Hom}_{\mathrm{cts}}
\bigl(H^i(L_n,\mathbb{Z}_p^p),\mathbb{Q}_p/\mathbb{Z}_p\bigr)
\cong H^{3-i}(L_n,\mathbb{Q}_p/\mathbb{Z}_p).
\]

\subsection{The cohomology of $L_n$ for $p>3$ and $R=\mathbb{Z}_p$}

For the rest of this section,
let $p>3$ and $R=\mathbb{Z}_p$.
We first compute $H^*(L_n,\mathbb{Z}_p^p)$.
Since $L_n$ is a $p$-adic Lie group of dimension $3$ with no $p$-torsion,
it is a Poincar\'e group of dimension $3$.
Its dualizing modules have trivial action,
so its cohomology with coefficients in $\mathbb{Z}_p^p$ is $\mathbb{Z}_p$ in degrees $0$ and $3$,
and vanishes in degrees at least $4$.
Poincar\'e duality gives
\begin{align*}
	H^1(L_n,\mathbb{Z}_p^p)
	&\cong H^2(L_n,\mathbb{Q}_p/\mathbb{Z}_p)^\vee,\\
	H^2(L_n,\mathbb{Z}_p^p)
	&\cong H^1(L_n,\mathbb{Q}_p/\mathbb{Z}_p)^\vee.
\end{align*}

We begin with the first cohomology for the four coefficient modules
\[
\mathbb{Z}_p^p,~
\mathbb{Z}/p^m\mathbb{Z},~
\mathbb{Q}_p/\mathbb{Z}_p,~
\mathbb{Z}_p^d.
\]
These groups are given by continuous homomorphisms from $L_n^{\mathrm{ab}}$ to the corresponding coefficient modules.
For every $n\geqq 0$, the $p$-primary component of $L_n^{\mathrm{ab}}$ is isomorphic to $F_n$.
Thus
\[
\operatorname{Hom}_{\mathrm{cts}}(L_n^{\mathrm{ab}},\mathbb{Z}_p^p)
\cong\operatorname{Hom}_{\mathrm{cts}}(F_n,\mathbb{Z}_p^p)=0,
\]
since $F_n$ is finite and $\mathbb{Z}_p$ is torsion-free.
The same argument applies to discrete coefficients:
\[
H^1(L_n,\mathbb{Z}_p^d)
\cong\operatorname{Hom}_{\mathrm{cts}}(L_n^{\mathrm{ab}},\mathbb{Z}_p^d)
=0.
\]
For finite coefficients,
every homomorphism factors through the
$p$-primary component,
and hence
\begin{align*}
	H^1(L_n,\mathbb{Z}/p^m\mathbb{Z})
	&\cong\operatorname{Hom}_{\mathrm{cts}}
	(F_n,\mathbb{Z}/p^m\mathbb{Z})\\
	&\cong F_{n,m}.
\end{align*}
Similarly,
\begin{align*}
	\operatorname{Hom}_{\mathrm{cts}}(L_n,\mathbb{Q}_p/\mathbb{Z}_p)
	&\cong\operatorname{Hom}_{\mathrm{cts}}(F_n,\mathbb{Q}_p/\mathbb{Z}_p)\\
	&=F_n^\vee\cong F_n,
\end{align*}
so
\[
H^1(L_n,\mathbb{Q}_p/\mathbb{Z}_p)\cong F_n.
\]

We now apply Poincar\'e duality.
Recall that the dualizing module $I_n$ is isomorphic to $\mathbb{Q}_p/\mathbb{Z}_p$ with trivial $L_n$-action.
For every $m\geqq 1$,
\[
(\mathbb{Z}/p^m\mathbb{Z})^\vee
:=\operatorname{Hom}(\mathbb{Z}/p^m\mathbb{Z},I_n)
\cong\mathbb{Z}/p^m\mathbb{Z}.
\]
Therefore
\begin{align*}
	H^2(L_n,\mathbb{Z}/p^m\mathbb{Z})
	&\cong H^1\bigl(L_n,(\mathbb{Z}/p^m\mathbb{Z})^\vee\bigr)^\vee\\
	&\cong H^1(L_n,\mathbb{Z}/p^m\mathbb{Z})^\vee\\
	&\cong F_{n,m}^\vee\\
	&\cong F_{n,m}.
\end{align*}
For profinite coefficients,
duality gives
\[
H^2(L_n,\mathbb{Z}_p^p)
\cong H^1(L_n,\mathbb{Q}_p/\mathbb{Z}_p)^\vee
\cong F_n.
\]
For coefficients in $\mathbb{Q}_p/\mathbb{Z}_p$,
we obtain
\[
H^2(L_n,\mathbb{Q}_p/\mathbb{Z}_p)
\cong H^1(L_n,\mathbb{Z}_p^p)^\vee=0.
\]
Finally,
the torsion subgroup of $H^2(L_n,\mathbb{Z}_p^p)$ is $F_n$,
which is zero for $n=0,1$ and isomorphic to $\mathbb{Z}/p^{n-1}\mathbb{Z}$ for $n\geqq 2$.
Since $H^1(L_n,\mathbb{Z}_p^p)=0$,
we obtain
\[
H^2(L_n,\mathbb{Z}_p^d)\cong F_n.
\]

Collecting the results for finite coefficients,
we have
\[
H^i(L_n,\mathbb{Z}/p^m\mathbb{Z})\cong
\begin{cases}
	\mathbb{Z}/p^m\mathbb{Z} & (i=0),\\
	F_{n,m} & (i=1,2),\\
	\mathbb{Z}/p^m\mathbb{Z} & (i=3),\\
	0 & (i\geqq 4).
\end{cases}
\]
For profinite coefficients,
\[
H^i(L_n,\mathbb{Z}_p^p)\cong
\begin{cases}
	\mathbb{Z}_p & (i=0),\\
	0 & (i=1),\\
	F_n & (i=2),\\
	\mathbb{Z}_p & (i=3),\\
	0 & (i\geqq 4).
\end{cases}
\]
In degree $3$,
duality also gives
\begin{align*}
	H^3(L_n,\mathbb{Q}_p/\mathbb{Z}_p)
	&\cong H^0(L_n,\mathbb{Z}_p^p)^\vee\\
	&\cong\mathbb{Z}_p^\vee\\
	&\cong\mathbb{Q}_p/\mathbb{Z}_p.
\end{align*}
Thus,
for every $n\geqq 0$,
\[
H^i(L_n,\mathbb{Q}_p/\mathbb{Z}_p)\cong
\begin{cases}
	\mathbb{Q}_p/\mathbb{Z}_p & (i=0),\\
	F_n & (i=1),\\
	0 & (i=2),\\
	\mathbb{Q}_p/\mathbb{Z}_p & (i=3),\\
	0 & (i\geqq 4).
\end{cases}
\]
For discrete $\mathbb{Z}_p$-coefficients,
we similarly obtain
\[
H^i(L_n,\mathbb{Z}_p^d)\cong
\begin{cases}
	\mathbb{Z}_p & (i=0),\\
	0 & (i=1),\\
	F_n & (i=2),\\
	0 & (i=3),\\
	\mathbb{Q}_p/\mathbb{Z}_p & (i=4),\\
	0 & (i\geqq 5).
\end{cases}
\]

\subsection{The cohomology of $K_n$}

Let $p>3$, $n\geqq 0$, and $m\geqq 1$.
For trivial coefficients,
we have
\[
H^1(K_n,-)=\operatorname{Hom}_{\mathrm{cts}}(K_n,-).
\]
The description of $K_n^{\mathrm{ab}}$ reduces the computation with $\mathbb{Z}_p^d$ and $\mathbb{Z}_p^p$ coefficients to
\[
\operatorname{Hom}_{\mathrm{cts}}(\mathbb{Z}_p^\times,\mathbb{Z}_p^d)=0,~
\operatorname{Hom}_{\mathrm{cts}}(\mathbb{Z}_p^\times,\mathbb{Z}_p^p)
\cong\mathbb{Z}_p^p.
\]
The pro-$p$ factor of $K_n^{\mathrm{ab}}$ is isomorphic to $\mathbb{Z}_p\times F_n$.
Hence
\[
\operatorname{Hom}_{\mathrm{cts}}(K_n,\mathbb{Z}/p^m\mathbb{Z})
\cong\mathbb{Z}/p^m\mathbb{Z}\times F_{n,m}.
\]
Using $\mathbb{Z}_p^\times\cong\mu_{p-1}\times(1+p\mathbb{Z}_p)$,
we also obtain
\begin{align*}
	H^1(K_n,\mathbb{Q}_p/\mathbb{Z}_p)
	&\cong\operatorname{Hom}_{\mathrm{cts}}\bigl(
	\mathbb{Z}_p^\times\times(\mathbb{Z}/p^n\mathbb{Z})^\times,
	\mathbb{Q}_p/\mathbb{Z}_p\bigr)\\
	&\cong\operatorname{Hom}_{\mathrm{cts}}
	(\mathbb{Z}_p,\mathbb{Q}_p/\mathbb{Z}_p)
	\times\operatorname{Hom}_{\mathrm{cts}}
	(F_n,\mathbb{Q}_p/\mathbb{Z}_p)\\
	&\cong\mathbb{Q}_p/\mathbb{Z}_p\times F_n.
\end{align*}

Since $p>3$,
the group $\operatorname{GL}_2(\mathbb{Z}_p)$ is a $p$-adic Lie group of dimension $4$ with no $p$-torsion.
Thus $K_n$ is a Poincar\'e group of dimension $4$.
As in the case of $L_n$,
its dualizing modules have trivial action.
Duality therefore gives
\begin{align*}
	H^4(K_n,\mathbb{Z}_p^p)
	&\cong H^0(K_n,\mathbb{Q}_p/\mathbb{Z}_p)^\vee
	\cong(\mathbb{Q}_p/\mathbb{Z}_p)^\vee
	\cong\mathbb{Z}_p,\\
	H^4(K_n,\mathbb{Q}_p/\mathbb{Z}_p)
	&\cong H^0(K_n,\mathbb{Z}_p^p)^\vee
	\cong(\mathbb{Z}_p^p)^\vee
	\cong\mathbb{Q}_p/\mathbb{Z}_p,\\
	H^4(K_n,\mathbb{Z}/p^m\mathbb{Z})
	&\cong H^0(K_n,\mathbb{Z}/p^m\mathbb{Z})^\vee
	\cong(\mathbb{Z}/p^m\mathbb{Z})^\vee
	\cong\mathbb{Z}/p^m\mathbb{Z}.
\end{align*}
In degree $3$,
we obtain
\begin{align*}
	H^3(K_n,\mathbb{Z}_p^p)
	&\cong H^1(K_n,\mathbb{Q}_p/\mathbb{Z}_p)^\vee\\
	&\cong(\mathbb{Q}_p/\mathbb{Z}_p)^\vee\times F_n^\vee\\
	&\cong\mathbb{Z}_p\times F_n,\\
	H^3(K_n,\mathbb{Q}_p/\mathbb{Z}_p)
	&\cong H^1(K_n,\mathbb{Z}_p^p)^\vee\\
	&\cong(\mathbb{Z}_p^p)^\vee\\
	&\cong\mathbb{Q}_p/\mathbb{Z}_p,\\
	H^3(K_n,\mathbb{Z}/p^m\mathbb{Z})
	&\cong H^1(K_n,\mathbb{Z}/p^m\mathbb{Z})^\vee\\
	&\cong\mathbb{Z}/p^m\mathbb{Z}\times F_{n,m}.
\end{align*}

It remains to compute the cohomology with coefficients in $\mathbb{Z}_p^d$ and the second cohomology with coefficients in $\mathbb{Z}_p^p$,
$\mathbb{Q}_p/\mathbb{Z}_p$,
and $\mathbb{Z}/p^m\mathbb{Z}$.
Since $K_n\cong L_n\times\mathbb{Z}_p$,
\cite[Theorem~2.4.6]{cohomology-of-number-fields} gives,
for a discrete coefficient module $A^d$ with trivial action,
\[
H^i(K_n,A^d)
\cong\bigoplus_{r+s=i}
H^r\bigl(L_n,H^s(\mathbb{Z}_p,A^d)\bigr).
\]
Recall that
\begin{align*}
	H^0(\mathbb{Z}_p,\mathbb{Z}_p^d)&\cong\mathbb{Z}_p,\\
	H^1(\mathbb{Z}_p,\mathbb{Z}_p^d)&=0,\\
	H^2(\mathbb{Z}_p,\mathbb{Z}_p^d)&\cong\mathbb{Q}_p/\mathbb{Z}_p,
\end{align*}
and
\[
H^0(\mathbb{Z}_p,\mathbb{Z}/p^m\mathbb{Z})
\cong H^1(\mathbb{Z}_p,\mathbb{Z}/p^m\mathbb{Z})
\cong\mathbb{Z}/p^m\mathbb{Z}.
\]

The cohomology of $K_n$ in degrees $0$ and $1$ has already been computed.
Since $H^i(\mathbb{Z}_p,\mathbb{Z}_p^d)=0$ for $i\geqq 3$,
we have
\begin{align*}
	H^2(K_n,\mathbb{Z}_p^d)
		&\cong H^0\bigl(L_n,H^2(\mathbb{Z}_p,\mathbb{Z}_p^d)\bigr)
				\oplus H^1\bigl(L_n,H^1(\mathbb{Z}_p,\mathbb{Z}_p^d)\bigr)
				\oplus H^2\bigl(L_n,H^0(\mathbb{Z}_p,\mathbb{Z}_p^d)\bigr)\\
		&\cong\mathbb{Q}_p/\mathbb{Z}_p\oplus0\oplus F_n\\
		&\cong\mathbb{Q}_p/\mathbb{Z}_p\oplus F_n,\\[4pt]
	H^3(K_n,\mathbb{Z}_p^d)
		&\cong H^1\bigl(L_n,H^2(\mathbb{Z}_p,\mathbb{Z}_p^d)\bigr)
			\oplus H^2\bigl(L_n,H^1(\mathbb{Z}_p,\mathbb{Z}_p^d)\bigr)
			\oplus H^3\bigl(L_n,H^0(\mathbb{Z}_p,\mathbb{Z}_p^d)\bigr)\\
		&\cong F_n\oplus0\oplus0\\
		&\cong F_n,\\[4pt]
	H^4(K_n,\mathbb{Z}_p^d)
		&\cong H^2\bigl(L_n,H^2(\mathbb{Z}_p,\mathbb{Z}_p^d)\bigr)
			\oplus H^3\bigl(L_n,H^1(\mathbb{Z}_p,\mathbb{Z}_p^d)\bigr)
			\oplus H^4\bigl(L_n,H^0(\mathbb{Z}_p,\mathbb{Z}_p^d)\bigr)\\
		&\cong0\oplus0\oplus\mathbb{Q}_p/\mathbb{Z}_p\\
		&\cong\mathbb{Q}_p/\mathbb{Z}_p,\\[4pt]
	H^5(K_n,\mathbb{Z}_p^d)
		&\cong H^3\bigl(L_n,H^2(\mathbb{Z}_p,\mathbb{Z}_p^d)\bigr)
			\oplus H^4\bigl(L_n,H^1(\mathbb{Z}_p,\mathbb{Z}_p^d)\bigr)
			\oplus H^5\bigl(L_n,H^0(\mathbb{Z}_p,\mathbb{Z}_p^d)\bigr)\\
		&\cong\mathbb{Q}_p/\mathbb{Z}_p\oplus0\oplus0\\
		&\cong\mathbb{Q}_p/\mathbb{Z}_p.
\end{align*}
Consequently,
\[
H^i(K_n,\mathbb{Z}_p^d)\cong
\begin{cases}
	\mathbb{Z}_p & (i=0),\\
	0 & (i=1),\\
	\mathbb{Q}_p/\mathbb{Z}_p\oplus F_n & (i=2),\\
	F_n & (i=3),\\
	\mathbb{Q}_p/\mathbb{Z}_p & (i=4,5),\\
	0 & (i\geqq 6).
\end{cases}
\]

For finite coefficients,
we use $H^i(\mathbb{Z}_p,\mathbb{Z}/p^m\mathbb{Z})=0$ for $i\geqq 2$.
This gives
\begin{align*}
	H^2(K_n,\mathbb{Z}/p^m\mathbb{Z})
		&\cong H^1
		\bigl(
			L_n,H^1(\mathbb{Z}_p,\mathbb{Z}/p^m\mathbb{Z})
		\bigr)
		\oplus H^2
		\bigl(
			L_n,H^0(\mathbb{Z}_p,\mathbb{Z}/p^m\mathbb{Z})
		\bigr)\\
	&\cong F_{n,m}\oplus F_{n,m}.
\end{align*}
Similarly,
since $H^i(\mathbb{Z}_p,\mathbb{Q}_p/\mathbb{Z}_p)=0$ for $i\geqq 2$,
\begin{align*}
	H^2(K_n,\mathbb{Q}_p/\mathbb{Z}_p)
		&\cong H^1
		\bigl(
			L_n,H^1(\mathbb{Z}_p,\mathbb{Q}_p/\mathbb{Z}_p)
		\bigr)
		\oplus H^2
		\bigl(
			L_n,H^0(\mathbb{Z}_p,\mathbb{Q}_p/\mathbb{Z}_p)
		\bigr).
\end{align*}
The first summand is isomorphic to $F_n$,
and the second is zero.
Thus
\[
H^2(K_n,\mathbb{Q}_p/\mathbb{Z}_p)\cong F_n.
\]
Finally,
Poincar\'e duality gives
\[
H^2(K_n,\mathbb{Z}_p^p)
\cong H^2(K_n,\mathbb{Q}_p/\mathbb{Z}_p)^\vee
\cong F_n^\vee\cong F_n.
\]
We have proved the following proposition.

\begin{proposition}
	Suppose that $p>3$,
	$n\geqq 0$,
	and $m\geqq 1$.
	Then the cohomology groups of $K_n$ are as follows:
	\begin{enumerate}
		\item For discrete $\mathbb{Z}_p$-coefficients,
		\[
		H^i(K_n,\mathbb{Z}_p^d)\cong
		\begin{cases}
			\mathbb{Z}_p & (i=0),\\
			0 & (i=1),\\
			\mathbb{Q}_p/\mathbb{Z}_p\oplus F_n & (i=2),\\
			F_n & (i=3),\\
			\mathbb{Q}_p/\mathbb{Z}_p & (i=4,5),\\
			0 & (i\geqq 6).
		\end{cases}
		\]
		\item For profinite $\mathbb{Z}_p$-coefficients,
		\[
		H^i(K_n,\mathbb{Z}_p^p)\cong
		\begin{cases}
			\mathbb{Z}_p & (i=0,1),\\
			F_n & (i=2),\\
			\mathbb{Z}_p\oplus F_n & (i=3),\\
			\mathbb{Z}_p & (i=4),\\
			0 & (i\geqq 5).
		\end{cases}
		\]
		\item For coefficients in $\mathbb{Q}_p/\mathbb{Z}_p$,
		\[
		H^i(K_n,\mathbb{Q}_p/\mathbb{Z}_p)\cong
		\begin{cases}
			\mathbb{Q}_p/\mathbb{Z}_p & (i=0),\\
			\mathbb{Q}_p/\mathbb{Z}_p\oplus F_n & (i=1),\\
			F_n & (i=2),\\
			\mathbb{Q}_p/\mathbb{Z}_p & (i=3,4),\\
			0 & (i\geqq 5).
		\end{cases}
		\]
		\item For finite coefficients,
		\[
		H^i(K_n,\mathbb{Z}/p^m\mathbb{Z})\cong
		\begin{cases}
			\mathbb{Z}/p^m\mathbb{Z} & (i=0),\\
			\mathbb{Z}/p^m\mathbb{Z}\oplus F_{n,m} & (i=1),\\
			F_{n,m}\oplus F_{n,m} & (i=2),\\
			\mathbb{Z}/p^m\mathbb{Z}\oplus F_{n,m} & (i=3),\\
			\mathbb{Z}/p^m\mathbb{Z} & (i=4),\\
			0 & (i\geqq 5).
		\end{cases}
		\]
	\end{enumerate}
\end{proposition}

\subsection{The derived Hecke algebra}

The preceding computations determine the additive structure of $\mathscr{H}(G,K)_{\mathbb{Z}_p}$.
We reindex the set $A$ of double coset representatives as
\[
A=\left\{
x_{b,n}=\begin{pmatrix}
	p^{n+b} & 0\\
	0 & p^b
\end{pmatrix}
\;\middle|\;
b\in\mathbb{Z},\quad n\geqq 0
\right\}.
\]
Here,
for an abelian group $M$,
the notation $M[A]=\bigoplus_{x\in A}Mx$ denotes the direct sum indexed by $A$.

\begin{proposition}\label{proposition5}
	There are isomorphisms of abelian groups
	\[
	\mathscr{H}^i(G,K)_{\mathbb{Z}_p}\cong
	\begin{cases}
		\mathbb{Z}_p[A] & (i=0),\\
		0 & (i=1),\\
		\displaystyle\bigoplus_{\substack{b\in\mathbb{Z}\\n\geqq 0}}
		(\mathbb{Q}_p/\mathbb{Z}_p\oplus F_n)x_{b,n} & (i=2),\\
		\displaystyle\bigoplus_{\substack{b\in\mathbb{Z}\\n\geqq 0}}
		F_nx_{b,n} & (i=3),\\
		(\mathbb{Q}_p/\mathbb{Z}_p)[A] & (i=4,5),\\
		0 & (i\geqq 6).
	\end{cases}
	\]
\end{proposition}

We next consider $\overline{\mathscr{H}}(G,K)_{\mathbb{Z}_p}$.
For every $n\geqq 0$,
$m\geqq 1$,
and $i\geqq 0$,
the group $H^i(K_n,\mathbb{Z}/p^m\mathbb{Z})$ is finite.
For fixed $n$ and $i$, the inverse system
\[
\bigl(H^i(K_n,\mathbb{Z}/p^m\mathbb{Z})\bigr)_{m\geqq 1}
\]
therefore satisfies the Mittag--Leffler condition.
Consequently,
\[
H^i(K_n,\mathbb{Z}_p^p)
\cong\varprojlim_{m\geqq 1}H^i(K_n,\mathbb{Z}/p^m\mathbb{Z}).
\]
This gives the following description of
$\overline{\mathscr{H}}(G,K)_{\mathbb{Z}_p}$.

\begin{proposition}
	With $A$ as above, there are isomorphisms of abelian groups
	\[
	\overline{\mathscr{H}}^i(G,K)_{\mathbb{Z}_p}\cong
	\begin{cases}
		\mathbb{Z}_p[A] & (i=0,1),\\
		\displaystyle\bigoplus_{\substack{b\in\mathbb{Z}\\n\geqq 0}}
		F_nx_{b,n} & (i=2),\\
		\displaystyle\bigoplus_{\substack{b\in\mathbb{Z}\\n\geqq 0}}
		(\mathbb{Z}_p\oplus F_n)x_{b,n} & (i=3),\\
		\mathbb{Z}_p[A] & (i=4),\\
		0 & (i\geqq 5).
	\end{cases}
	\]
\end{proposition}

Finally,
for every $i\geqq 0$, the group $\widehat{\mathscr{H}}^i(G,K)_{\mathbb{Z}_p}$ is the $p$-adic completion of $\overline{\mathscr{H}}^i(G,K)_{\mathbb{Z}_p}$.

% !TeX root = main.tex


\begin{thebibliography}{99}
	\bibitem{stacks-project}
	The Stacks Project Authors. 
	\textit{Stacks Project}. 
	\url{https://stacks.math.columbia.edu/}, 2026.
	\bibitem{Venkatesh}
	A.~Venkatesh,
	\emph{Derived Hecke algebra and cohomology of arithmetic groups},
	Forum Math. Pi \textbf{7} (2019), e7, 119 pp.,
	doi:10.1017/fmp.2019.6.
	\bibitem{cohomology-of-number-fields}
	J. Neukirch, A. Schmidt, and K. Wingberg. 
	\textit{Cohomology of Number Fields}, Second Edition. 
	Grundlehren der mathematischen Wissenschaften, 323. Springer-Verlag, Berlin, 2008.
	\bibitem{koziol2024parahoricheckeextalgebrascharacteristic}
	K.~Kozio{\l}, R.~Ollivier, and J.~Stockton,
	\emph{Parahoric Hecke Ext-algebras in characteristic $p$},
	Trans. Amer. Math. Soc. \textbf{379} (2026), no.~9, 6297--6363,
	doi:10.1090/tran/9663.
	\bibitem{boggi2016continuous}
	M.~Boggi and G.~Corob Cook,
	\emph{Continuous cohomology and homology of profinite groups},
	Documenta Mathematica \textbf{21} (2016), 1269--1312.
	\bibitem{venjakob2002structure}
	O.~Venjakob,
	\emph{On the structure theory of the Iwasawa algebra of a $p$-adic Lie group},
	Journal of the European Mathematical Society \textbf{4} (2002),
	no.~3, 271--311.
	\bibitem{beaudry2022dualizing}
	A.~Beaudry, P.~G.~Goerss, M.~J.~Hopkins, and V.~Stojanoska,
	\emph{Dualizing spheres for compact $p$-adic analytic groups
		and duality in chromatic homotopy},
	Invent. Math. \textbf{229} (2022), no.~3, 1301--1434.
	\bibitem{kohlhaase2017smooth}
	J.~Kohlhaase,
	\emph{Smooth duality in natural characteristic},
	Adv. Math. \textbf{317} (2017), 1--49.
\end{thebibliography}
\end{document}